\documentclass[10pt,a4paper]{amsart}
\usepackage[utf8]{inputenc}
\usepackage{amsmath,amssymb,mathrsfs,amsfonts,mathtools}
\usepackage{mathtools}
\usepackage{graphicx,geometry}
\usepackage{setspace}
\usepackage{float,comment,xcolor}
\graphicspath{{plot/}}
\newtheorem{theorem}{Theorem}[section]
\newtheorem{lemma}[theorem]{Lemma}

\newtheorem{prop}[theorem]{Proposition}
\newtheorem{cor}[theorem]{Corollary}
\newtheorem{fact}{Fact}

\theoremstyle{definition}
\newtheorem{definition}[theorem]{Definition}

\theoremstyle{remark}
\newtheorem{remark}[theorem]{Remark}

\numberwithin{equation}{section}

\title {Hypercomplex structures arising from twistor spaces}
\author{Shuo Wang}
\address{School of Mathematical  Sciences, University of Science and Technology of China, Hefei 230026 China}
\email{ws220122@mail.ustc.edu.cn}
\thanks{ B.X. is supported in part by the Project of Stable Support for Youth Team in Basic Research Field, CAS (Grant No. YSBR-001) and 
NSFC (Grant Nos. 12271495, 11971450 and 12071449).  }
\thanks{$^\dagger$B.X. is the corresponding author.}
\author{Bin Xu}
\address{CAS Wu Wen-Tsun Key Laboratory of Mathematics and School of Mathematical  \newline \indent Sciences, University of Science and Technology of China, Hefei 230026 China}
\email{bxu@ustc.edu.cn}

\begin{document}

\maketitle


\begin{abstract}
A hyperkähler manifold is defined as a Riemannian manifold endowed with three covariantly constant complex structures that are quaternionically related. A twistor space is characterized as a holomorphic fiber bundle $p: \mathcal{Z} \rightarrow \mathbb{CP}^1$ possesses properties such as a family of holomorphic sections whose normal bundle is $\bigoplus^{2n}\mathcal{O}(1)$, a holomorphic section of $\Lambda^2(N\mathcal{Z})\otimes p^*(\mathcal{O}(2))$ that defines a symplectic form on each fiber, and a compatible real structure. According to the Hitchin-Karlhede-Lindstr{\"o}m-Ro{\v{c}}ek theorem ({\it Comm. Math. Phys.}, 108(4):535-589, 1987), there exists a hyperkähler metric on the parameter space $M$ for the real sections of $\mathcal{Z}$.
Utilizing the Kodaira-Spencer deformation theory, we facilitate the construction of a hypercomplex structure on \(M\), predicated upon more relaxed presuppositions concerning \(\mathcal{Z}\).
This effort enriches our understanding of the classical theorem by Hitchin-Karlhede-Lindstr{\"o}m-Ro{\v{c}}ek.
\end{abstract}


\section{Introduction}

A \textit{hypercomplex manifold} $(M, \mathbf{I}, \mathbf{J}, \mathbf{K})$ is a smooth manifold $M$ equipped with three complex structures $\mathbf{I}$, $\mathbf{J}$, and $\mathbf{K}$ that satisfy the quaternionic relations: 
\[
\mathbf{I}\circ\mathbf{J} = -\mathbf{J}\circ\mathbf{I} = \mathbf{K}.
\]
A \textit{hyperkähler manifold} $(M, g, \mathbf{I}, \mathbf{J}, \mathbf{K})$ is a hypercomplex manifold $(M, \mathbf{I}, \mathbf{J}, \mathbf{K})$ endowed with a Riemannian metric $g$ on $M$, such that $(M, g)$ becomes a Kähler manifold with respect to the three complex structures $\mathbf{I}$, $\mathbf{J}$, and $\mathbf{K}$, each being covariantly constant under the Levi-Civita connection induced by $g$.

Let $(M, g, \mathbf{I}, \mathbf{J}, \mathbf{K})$ be a hyperkähler manifold. We define $\mathbf{I}_\zeta := a\mathbf{I} + b\mathbf{J} + c\mathbf{K}$ where 
\[(a, b, c) = \left(\frac{1-\zeta\bar\zeta}{1+\zeta\bar\zeta}, \frac{\zeta+\bar\zeta}{1+\zeta\bar\zeta}, \frac{i(-\zeta+\bar\zeta)}{1+\zeta\bar\zeta}\right) \in S^2\] 
for each $\zeta \in \mathbb{C} \cup \{\infty\} \cong \mathbb{CP}^1$. Consequently, the set $\{\mathbf{I}_\zeta | \zeta \in \mathbb{CP}^1\}$ forms a collection of complex structures on $M$, and $(M, g, \mathbf{I}_\zeta)$ constitutes a Kähler manifold for each $\zeta \in \mathbb{CP}^1$. Let $\omega_1 = g(\mathbf{I}\cdot,\cdot)$, $\omega_2 = g(\mathbf{J}\cdot,\cdot)$, and $\omega_3 = g(\mathbf{K}\cdot,\cdot)$.

From a hyperkähler manifold $M$ of real dimension $4n$, we define the \textit{twistor space}, denoted as $\mathcal{Z}$, as the smooth product manifold $\mathcal{Z} = M \times \mathbb{CP}^1$ equipped with an almost complex structure $\underline{\mathbf{I}}$. In $T_{(m,\zeta)}\mathcal{Z}$, the almost complex structure $\underline{\mathbf{I}}$ is given by $(\mathbf{I}_\zeta,\mathbf{I}_0)$, where $\mathbf{I}_0$ represents the operation of multiplying by $\sqrt{-1}$ on the tangent space of $\mathbb{CP}^1$. Hitchin, Karlhede, Lindström, and Roček \cite[Theorem 3.3]{hitchin1987hyperkahler} have shown that the almost complex structure $\underline{\mathbf{I}}$ constitutes a complex structure on $\mathcal{Z}$. Moreover, the natural projection $p: \mathcal{Z} \rightarrow \mathbb{CP}^1$ is a complex fiber bundle over $\mathbb{CP}^1$, whose fiber over $\zeta \in \mathbb{CP}^1$ is the Kähler manifold $(M, g, \mathbf{I}_\zeta)$, denoted as $\mathcal{Z}_\zeta$. In simpler terms, for each $x \in M$, there exists a complex projective line $\mathbb{CP}^1_x$ within $\mathcal{Z}$ associated with the point $x \in M$, such that: 
\[\mathcal{Z} = \bigcup_{\zeta \in \mathbb{CP}^1}(M, g, \mathbf{I}_\zeta) = \bigcup_{x \in M} \mathbb{CP}_x^1.\]
Furthermore, Hitchin-Karlhede-Lindstr{\"o}m-Ro{\v{c}}ek's theorem \cite[Theorem 3.3]{hitchin1987hyperkahler} states the following:
\begin{enumerate}
\item[(\romannumeral 1)] Each $\mathbb{CP}_x^1$ has a normal bundle that is isomorphic to $\bigoplus^{2n}\mathcal{O}(1)$.
\item[(\romannumeral 2)] The holomorphic section $\omega$ of $\Lambda^2(N\mathcal{Z})^*\otimes p^*(\mathcal{O}(2))$ is defined as follows: 
\[
\omega=
\begin{cases}
(\omega_2+i\omega_3)+2\zeta_0\omega_1-\zeta_0^2(\omega_2-i\omega_3) & \text{for } \zeta_0\in U_0\\
\zeta_1^2(\omega_2+i\omega_3)+2\zeta_1\omega_1-(\omega_2-i\omega_3) & \text{for } \zeta_1\in U_1
\end{cases}
\]
where $(U_0,\zeta_0)\cup(U_1,\zeta_1)/\sim$ are open charts on $\mathbb{CP}^1$ as defined in the second paragraph of Section 3 and $N\mathcal{Z}$ represents the vertical space of the fiber bundle $p$. In particular, $\omega|_{\mathcal{Z}_\zeta}$ is a symplectic form on the fiber $(M,g,I_\zeta)$ for each $\zeta\in\mathbb{CP}^1$.
\item[(\romannumeral 3)] There is an involution mapping $\tau$ on $\mathcal{Z}$ induced by the antipodal map on each $\mathbb{CP}_x^1$.
\end{enumerate}

Conversely, we designate $\mathcal{Z}$ as a \textit{twistor space} if it is a complex manifold of dimension $2n+1$ that satisfies the following conditions (1), (2), (3), and (4):
\begin{enumerate}
    \item $\mathcal{Z}$ is a holomorphic fiber bundle over $\mathbb{CP}^1$, denoted by $p: \mathcal{Z} \rightarrow \mathbb{CP}^1$.
    \item There exists a real structure $\tau$ on $\mathcal{Z}$ that is compatible with the antipodal map on $\mathbb{CP}^1$.
    \item $\mathcal{Z}$ admits a family $\mathcal{M}$ of holomorphic sections, where the normal bundle $T_Fs$ is isomorphic to $\bigoplus^{2n}\mathcal{O}(1)$ for each $s \in \mathcal{M}$.
    \item[(3$^\prime$)] $\mathcal{Z}$ admits a holomorphic section $s$ whose normal bundle is $T_Fs \cong \bigoplus^{2n}\mathcal{O}(1)$.
    \item There exists a holomorphic section $\omega$ of $\Lambda^2{(N\mathcal{Z})}^*\otimes p^*\mathcal{O}(2)$ such that:
    \begin{enumerate}
        \item $\omega$ defines a symplectic form on each fiber of $p: \mathcal{Z} \rightarrow \mathbb{CP}^1$.
        \item For each real section $s \in \mathcal{M}$, $\omega(\cdot, j(\tau, s)\cdot)$ is a negative definite quadratic form on $H^0\left(s(\mathbb{CP}^1), T_Fs\otimes\mathcal{O}(-1)\right) \cong \mathbb{C}^{2n}$.
    \end{enumerate}
\end{enumerate}
where $j(\tau, s)$ is a quaternionic structure on $H^0\left(s(\mathbb{CP}^1), T_Fs\otimes\mathcal{O}(-1)\right)$ induced by $\tau$. In \cite[Theorem 3.3]{hitchin1987hyperkahler}, Hitchin, Karlhede, Lindström, and Roček proposed a method for defining a hyperkähler manifold $M$ of real dimension $4n$ as the parameter space of real sections of the fiber bundle $p: \mathcal{Z} \rightarrow \mathbb{CP}^1$.
\begin{theorem}\cite[Theorem 3.3]{hitchin1987hyperkahler}
Let $\mathcal{Z}$ be a complex manifold and $\sigma$ the antipodal map over $\mathbb{CP}^1$. Assuming that conditions (1), (2), (3) and (4) are satisfied, and let $M$ be the set of real holomorphic sections in $\mathcal{M}$. If $M \neq \emptyset$, then it constitutes a smooth submanifold of $\mathcal{M}$. Furthermore, a hyperk\"ahler structure exists on $M$.
\end{theorem}
This establishes a bi-directional correspondence between hyperkähler manifolds and holomorphic fiber bundles over $\mathbb{CP}^1$.

The primary aim of this manuscript is to deepen the understanding of the theorem posited by Hitchin, Karlhede, Lindström, and Roček \cite{hitchin1987hyperkahler}, specifically by incorporating condition (3) into the theoretical framework originally developed by Kodaira \cite{kodaira1962theorem}. In this study, we explore a complex manifold $W$ and its compact complex submanifolds $V$, following the foundational research by Kodaira \cite{kodaira1962theorem}, which investigates the deformation of complex submanifolds within the overarching space $W$. These changes are described by paramater space which is a complex manifold, usually considered simply as $\mathbb{B} = \big\{t \in \mathbb{C}^l \mid |t|_{\infty} < \varepsilon\big\}$, where $\varepsilon$ is a tiny positive number. As outlined in Definition \ref{analytic family of compact submanifolds}, an\textit{ analytic family of compact complex submanifolds} is characterized through a mapping adhering to four crucial conditions. These conditions ensure that every submanifold in the family maintains a connected and compact complex structure with a defined dimension, and is closely connected to the parameter space via a complex submanifold of higher dimension. This concept is encapsulated in the natural projection $\pi: \mathcal{V} \rightarrow \mathbb{B}$, where $\mathcal{V}$ denotes a complex submanifold of $W \times \mathbb{B}$. Here, for each $t \in \mathbb{B}$, $V_t := \pi^{-1}(t)$ is identified as a compact complex submanifold of $W$, with $V = V_0$ marking the initial condition for the deformation. Infinitesimal deformation of $V_t$ in $W$ is intuitively understood as a holomorphic section of the normal bundle $NV_t$ of $V_t$. This concept is rigorously defined by Kodaira \cite{kodaira1962theorem} through the introduction of the Kodaira-Spencer map $\sigma: T_t\mathbb{B} \rightarrow H^0(V_t, NV_t)$. This map associates tangent vectors in the parameter space with holomorphic sections of the normal bundle, capturing the deformation's essence. A cornerstone of Kodaira's findings is the assertion that, provided the cohomology space $H^1(V, NV)$ vanishes, an analytic family $\pi: \mathcal{V} \rightarrow \mathbb{B}$ of compact complex submanifolds $V$ in $W$ exists for a sufficiently small parameter space $\mathbb{B}$. This family is distinguished by the Kodaira-Spencer map being a linear isomorphism, rendering the deformation maximal in this localized context. These characteristics lay the groundwork for subsequent analyses, herein referred to as the canonical local deformation of $V$. For a comprehensive discussion and formal definitions, refer to Theorem \ref{main Kod1} and Definition \ref{maximal}.

Building upon Kodaira \cite{kodaira1962theorem}, this work use the concept of \textit{regular deformation space}, denoted as $\widetilde{\mathscr{M}_X}$. The notion of a regular deformation space emerges as an intuitively natural framework for exploring the deformation processes of compact complex submanifolds. However, the precise origins of this concept within the scholarly literature remain unclear to us. This space is envisioned as a refined parameter space that facilitates a deeper exploration of the deformation process for a compact complex submanifold $X$ within a complex manifold $W$. Through this endeavor, we aim to modestly augment Kodaira's original framework, offering additional insights into the complexities of deformation. Specifically, let $X$ be a compact complex submanifold of $W$. A compact complex submanifold $V$ of $W$ is termed a \textit{deformation} of $X$ if it is possible to transition from $X$ to $V$ via a sequence of local deformations in a finite number of steps. To elaborate, a compact complex submanifold $V$ of $W$ constitutes a deformation of $X$ in $W$ if there exist analytic families $\pi_k \colon \mathcal{V}_k \to \mathbb{B}_k$ for $k = 0, \ldots, n-1$, characterized by the properties: $X = \pi_0^{-1}(s_0)$, $\pi_{k+1}^{-1}(s_{k+1}) = \pi_k^{-1}(t_k)$, and $V = \pi_{n-1}^{-1}(t_{n-1})$ for certain $s_k, t_k \in \mathbb{B}_k$. The collection of all such deformations of $X$ in $W$ is represented as $\mathscr{M}_X$.

Under the assumption that $H^1(X, NX) = 0$, the \textit{regular deformation space} of $X$ in $W$, symbolized as $\widetilde{\mathscr{M}_X}$, is conceptualized as the subset of $\mathscr{M}_X$ comprising those deformations $V$ for which $H^1(V, NV) = 0$. This is formally expressed as:
\[
\widetilde{\mathscr{M}_X} = \left\{ V \in \mathscr{M}_X \middle| H^1(V, NV) = 0 \right\}.
\]
For each deformation $V$ within $\mathscr{M}_X$, considering the collection of all canonical local deformations of $V$ as its open neighborhoods establishes a topology $\mathscr{S}_X$ on this ensemble. More details are discussed in Lemma \ref{topo}. We then define the space $\widetilde{\mathscr{M}_X^l}$ as:
\[
\widetilde{\mathscr{M}_X^l}=\left\{V\in\widetilde{\mathscr{M}_X}\middle|H^0(V,NV)=l\right\},
\]
Within the structured framework of our analysis, $\widetilde{\mathscr{M}_X^l}$ is identified as an open subset of the topological space $\left(\widetilde{\mathscr{M}_X}, \mathscr{S}_X\right)$. Furthermore, the construction of $\widetilde{\mathscr{M}_X}$ is meticulously achieved through the disjoint union $\widetilde{\mathscr{M}_X} = \bigsqcup\limits_{l \in \mathbb{Z}_{\geq 0}} \widetilde{\mathscr{M}_X^l}$, each component inheriting the topology induced from $\widetilde{\mathscr{M}_X}$.  

\begin{theorem} {\rm (Theorem \ref{complex regular deformation space})}
The set $\left(\widetilde{\mathscr{M}^l_X}, \mathscr{S}_X\right)$ is endowed with a complex manifold structure of complex dimension $l$, provided that $\widetilde{\mathscr{M}^l_X}$ is non-empty, for every $l \in \mathbb{Z}_{\geq0}$.
\end{theorem}

By utilizing condition the condition (1) that $p:\mathcal{Z}\rightarrow\mathbb{CP}^1$ is a holomorphic fiber bundle and suppose $s$ is a holomorphic section of $Z$ given by condition ($3^\prime$), we can regard $s(\mathbb{CP}^1)$ as a compact complex submanifold in $\mathcal{Z}$. Leveraging the framework we established earlier, it is possible to define a complex manifold $\widetilde{\mathscr{M}_{s(\mathbb{CP}^1)}^{4n}}$ included  in the regular defromation space of $s(\mathbb{CP}^1)$ in $\mathcal{Z}$. Then define $\mathcal{M}$, which serves as an open submanifold $\widetilde{\mathscr{M}_{s(\mathbb{CP}^1)}^{4n}}$ as
\[
\mathcal{M}:=\left\{V \in \widetilde{\mathscr{M}^{4n}_{s(\mathbb{CP}^1)}} | V \text{ is a holomorphic section of } \mathcal{Z}; T_FV \cong \bigoplus^{2n}\mathcal{O}(1)\right\}
\]
The manifold $\mathcal{M}$ acts as the parameter space for holomorphic sections of $\mathcal{Z}$. In this case, we can prove that $\mathcal{Z}$ satisfies condition (3), i.e. (1)+($3^\prime$)$\Rightarrow$(3). The details of this process are further elaborated in Proposition \ref{open submanifold}. Furthermore, conditions (1), (2), and (3)$^\prime$ collectively facilitate the identification of a smooth manifold $M$, endowed with a hypercomplex structure, as a smooth submanifold of $\mathcal{M}$. This fact is corroborated by Theorem \ref{hypercomplex}. The hypercomplex structure on $M$ is further elaborated in this article, extending beyond the scope of Hitchin, Karlhede, Lindström, and Roček's original paper. Joyce also mentioned the existence of twistor construction for hypercomplex manifolds in \cite[Subsection 7.5.1]{joyce2000compact}. The details will be discussed in Theorem \ref{hypercomplex} as follows:

\begin{theorem}{\rm (Theorem \ref{hypercomplex})}
Consider a complex manifold $\mathcal{Z}$ equipped with an antipodal map $\sigma$ over the projective line $\mathbb{CP}^1$. Upon imposing conditions (1), (2), and (3) for $\mathcal{Z}$, let $M$ denote the collection of real holomorphic sections within $\mathcal{M}$. Should $M$ be non-empty, it inherently forms a smooth submanifold of $\mathcal{M}$, further endowed with a hypercomplex structure.
\end{theorem}

We conclude the introduction by providing a comprehensive overview of the subsequent four sections of this manuscript. Section 2 delves into the deformation theory of compact complex submanifolds in complex manifolds, extending the foundational work of Kodaira \cite{kodaira1962theorem}. This section also introduces Lemma \ref{complex regular deformation space}, defining the regular deformation space within the realm of complex manifolds. In Section 3, the focus shifts to examining a real structure on $\bigoplus^{2n}\mathcal{O}(1)$. Proposition \ref{main section2} introduces a real structure on $H^0\big(\mathbb{CP}^1, \bigoplus^{2n}\mathcal{O}(1)\big)$ and delineates a quaternionic structure on $H^0\big(\mathbb{CP}^1, \mathbb{CP}^1\times\mathbb{C}^{2n}\big)$. Section 4 presents the parameter space $\mathcal{M}$, comprising holomorphic sections with the normal bundle $\bigoplus^{2n}\mathcal{O}(1)$. Utilizing theories from Section 2, $\mathcal{M}$ is established as a complex manifold. This section also demonstrates how the real structure $r$, introduced in Section 3, leads to the emergence of a submanifold $M$ within $\mathcal{M}$, and how the quaternionic structure $j$ bestows a hyperkähler structure on $M$. Section 5 sets forth a toy example to illustrate the relationship between hyperkähler manifolds $\mathbb{C}^{2n}$ and holomorphic fiber bundles over $\bigoplus^{2n}\mathcal{O}(1)$ on $\mathbb{CP}^1$.


\section{Deformation of compact complex submanifolds}

In this section, we will apply Kodaira's framework as delineated in \cite{kodaira1962theorem} to introduce the concept of an analytic family of compact submanifolds within a complex manifold $W$. Consider $X$ as a compact submanifold of $W$, meeting specific criteria. We aim to define a regular deformation space of $X$ in $W$, denoted as $\widetilde{\mathscr{M}_X}$, and to establish a complex structure on this space. This method aligns with the theory of deformation of complex analytic structures, initially developed by Kodaira and Spencer in their pioneering work \cite{kodaira1958deformations}.

Kodaira and Spencer, in their seminal article \cite{kodaira1958deformations}, established a foundational framework for studying deformation in complex manifolds. In the following sections, we will revisit and utilize several key definitions from their work.

\begin{definition}
\label{differentiable family of compact complex manifolds}
A \textit{differentiable family of compact complex manifolds} is defined as a differentiable fiber bundle $\pi\colon\mathcal{V}\rightarrow N$ over a differentiable manifold $N$, where each point of $\mathcal{V}$ has a neighborhood $\mathscr{U}\subset\mathcal{V}$ satisfying the following conditions: there exists a diffeomorphism $h\colon\mathscr{U}\rightarrow\mathbb{C}^n\times\pi(\mathscr{U})$ such that, for each point $t\in\pi(\mathscr{U})$, $V_t \coloneqq \pi^{-1}(t)$ is a compact complex manifold and the restriction of $h$ to $\mathscr{U}\cap V_t$ sends $\mathscr{U}\cap V_t\hookrightarrow\mathbb{C}^n\times\{t\}$ holomorphically, where $\mathbb{C}^n$ denotes the space of $n$ complex variables $\left(z^1,\cdots,z^n\right)$ and $n$ is the complex dimension of $V_t$. A \textit{complex analytic family of compact complex manifolds} is defined as a differentiable family of compact complex manifolds $\pi\colon\mathcal{V}\rightarrow N$ where $\pi\colon\mathcal{V}\rightarrow N$ is a holomorphic map and both $N$ and $\mathcal{V}$ are connected complex manifolds.
\end{definition}

\begin{definition}
\label{differentiable family of pi}
Let $\pi\colon\mathcal{V}\rightarrow N$ be a differentiable family of compact complex manifolds. A \textit{differentiable family of $\pi$} refers to a differentiable fiber bundle $\omega\colon\mathcal{B}\rightarrow\mathcal{V}$ with a structure group that is a complex Lie group $G$. The fiber bundle $\omega_t\colon B_t\rightarrow V_t$ is holomorphic when $\omega$ is restricted to each fiber $V_t \coloneqq \pi^{-1}(t)$, $t\in N$ of $\mathcal{V}$ and $B_t \coloneqq \omega^{-1}(V_t)$.

Let $\pi\colon\mathcal{V}\rightarrow N$ be a complex analytic family of compact complex manifolds. A \textit{complex analytic family of $\pi$} is defined as a differentiable family $\mathcal{B}\stackrel{\omega}{\rightarrow}\mathcal{V}\stackrel{\pi}{\rightarrow}N$ of $\pi$ that satisfies the condition of $\mathcal{B}\stackrel{\omega}{\rightarrow}\mathcal{V}$ being a holomorphic map.
\end{definition}

Utilizing the theoretical framework established in \cite{kodaira1958deformations}, Kodaira and Spencer formulated a pivotal theorem, the 'Principle of Upper Semi-continuity.' This theorem is elaborated and demonstrated in Kodaira's subsequent paper \cite{kodaira1957variation}.

\begin{theorem}
\label{Principle of Upper Semi-continuity}
Let $\pi\colon\mathcal{V}\rightarrow N$ be a differentiable family of compact complex manifolds as defined in \ref{differentiable family of compact complex manifolds}, and suppose the complex vector bundle $\omega\colon\mathcal{B}\rightarrow\mathcal{V}$ over $\mathcal{V}$ is also a differentiable family with respect to $\pi$ as defined in \ref{differentiable family of pi}. For each $t\in N$, let $V_t \coloneqq \pi^{-1}(t)$ and $B_t \coloneqq \omega^{-1}(V_t)$. Then there exists a neighborhood $U$ of $t\in N$ such that, for any $t^\prime\in U$ and $q\in\mathbb{Z}$,
\begin{equation*}
\dim H^q\left(V_{t^\prime}, B_{t^\prime}\right) \leq \dim H^q\left(V_{t}, B_{t}\right)
\end{equation*}
\end{theorem}

Kodaira's seminal paper \cite{kodaira1962theorem} introduces a critical concept in the study of complex manifolds: the analytic family of compact submanifolds. This concept is fundamental to understanding various aspects of compact complex submanifolds and is defined as follows:

\begin{definition}
\label{analytic family of compact submanifolds}
Let $W$ and $N$ be complex manifolds with complex dimensions $d+r$ and $l$, respectively. Let $j_1\colon W\times N\rightarrow W$ and $j_2\colon W\times N\rightarrow N$ be the natural projections. An \textit{analytic family of compact submanifolds} of $W$ with dimension $d$ and parameter space $N$ is a map $\pi\colon\mathcal{V}\rightarrow N$ satisfying the following conditions:
\begin{enumerate}
    \item $\mathcal{V}$ is a complex submanifold of $W\times N$ with codimension $r$.
    \item $\pi$ is the restriction of the map $j_2$ to $\mathcal{V}$.
    \item For each $t\in N$, $V_t \coloneqq \mathcal{V} \cap (W\times\{t\})$ is a connected and compact complex submanifold of $W\times\{t\}$ with dimension $d$.
    \item For each $p\in\mathcal{V}$, there exists a local chart
    \[
    p\in\Big(\mathscr{U}_p;w^1,\cdots, w^r, z^1,\cdots, z^d, t^1,\cdots, t^l\Big)\subset W\times N
    \]
    and holomorphic functions $f_1,\cdots, f_r$ on $\mathscr{U}_p$, such that
    \begin{align*}
    &(i) \mathcal{V}\cap\mathscr{U}_p=\Big\{(z,w,t)\in\mathscr{U}_p\colon f_1(z,w,t)=\cdots=f_r(z,w,t)=0\Big\}\\
    &(ii) \left.\textrm{rank}\frac{\partial(f_1, \cdots, f_r)}{\partial(w^1, \cdots, w^r,z^1,\cdots,z^d)}\right|_{p'}=r,\quad\forall p'\in\mathscr{U}_p
    \end{align*}
\end{enumerate}
\end{definition}

Next, we aim to establish the following proposition, which connects the analytic family of compact submanifolds, defined in Definition \ref{analytic family of compact submanifolds}, with the concept of a differentiable family of compact complex manifolds, as detailed in Definition \ref{differentiable family of compact complex manifolds}. This connection allows us to apply the Principle of Upper Semi-continuity (Theorem \ref{Principle of Upper Semi-continuity}) to our specific scenario.

\begin{prop}
Given an analytic family of compact submanifolds $\pi\colon\mathcal{V}\rightarrow N$ of a complex manifold $W$ with dimension $d$ and parameter space $N$, it follows that $\pi\colon\mathcal{V}\rightarrow N$ also constitutes a complex analytic family of compact complex manifolds.
\end{prop}

\begin{proof} 
Conditions (1), (2), and (3) imply that $\pi\colon\mathcal{V}\rightarrow N$ is a holomorphic surjective map. Without loss of generality, we can assume that $(z,w,t)(p)=0$ in condition (4). By the Implicit Function Theorem, condition (4-$ii$) implies that for any fixed $t_0\in j_2(\mathscr{U}_p)$ near $0$, there exists a unique simultaneous equation:
\[
\begin{cases}
w^1=\varphi^1(z^1,\cdots, z^d,t_0)\\
\vdots\\
w^r=\varphi^r(z^1,\cdots, z^d,t_0)
\end{cases}
\quad \text{such that}\quad
\begin{cases}
f_1(\varphi^1, \cdots, \varphi^r, z^1, \cdots, z^d,t_0)=0\\
\vdots\\
f_r(\varphi^1, \cdots, \varphi^r, z^1, \cdots, z^d,t_0)=0
\end{cases}
\]

for each $(z^1, \cdots, z^d)$ near $0$. Shrink $\mathscr{U}_p$ if necessary, then $\left(\mathscr{U}_p\cap V_{t_0},z^{r+1},\cdots, z^{r+d}\right)$ is a holomorphic chart on $V_{t_0}$ and $\left(\mathscr{U}_p\cap V_{t_0},z^{r+1},\cdots, z^{r+d}, t^1, \cdots, t^l\right)$ is a holomorphic chart on $\mathcal{V}$. The restriction of $\pi$ to $\mathscr{U}_p\cap\mathcal{V}$ is:
\[
\pi\colon\mathscr{U}_p\cap\mathcal{V}\rightarrow j_2(\mathscr{U}_p)\quad \Big(z^{r+1}, \cdots, z^{r+d}, t^1, \cdots, t^l\Big)\mapsto\Big(t^1, \cdots, t^l\Big)
\]
which implies that $\pi$ is a submersion. Furthermore, according to the Ehresmann fibration theorem \cite{ehresmann1950connexions}, each fiber of $\pi$ is a compact complex manifold, indicating that $\pi\colon\mathcal{V}\rightarrow N$ forms a differentiable fiber bundle. 
\end{proof}

In the context of exploring the complex manifold $W$ and its compact complex submanifold $V$, it becomes essential to define a specific system of local charts. These charts facilitate the examination of the manifold's local structure in the vicinity of $V$. We formalize this concept as follows:

\begin{definition}
\label{system of canonical local charts}
Define the \textit{system of canonical local charts of $W$ in the neighborhood of $V$} as a finite collection of charts $\Big\{\big(\mathscr{W}_k; w_k^1, \cdots, w_k^r, z_k^1, \cdots, z_k^d\big)\Big\}_{k\in I}$ on $W$. This collection satisfies the following conditions:
\begin{enumerate}
    \item $\bigcup_{k\in I} \mathscr{W}_k$ forms an open covering of $V$ in $W$.
    \item For each $k\in I$, the intersection $V\cap \mathscr{W}_k$ is defined by the conditions 
\[w^1_k = \cdots = w^r_k = 0.\]
\end{enumerate}
\end{definition}
This definition provides a framework for examining the local properties of $W$ around the submanifold $V$.

\begin{figure}[htbp]
    \centering
    \includegraphics[width=0.7\textwidth]{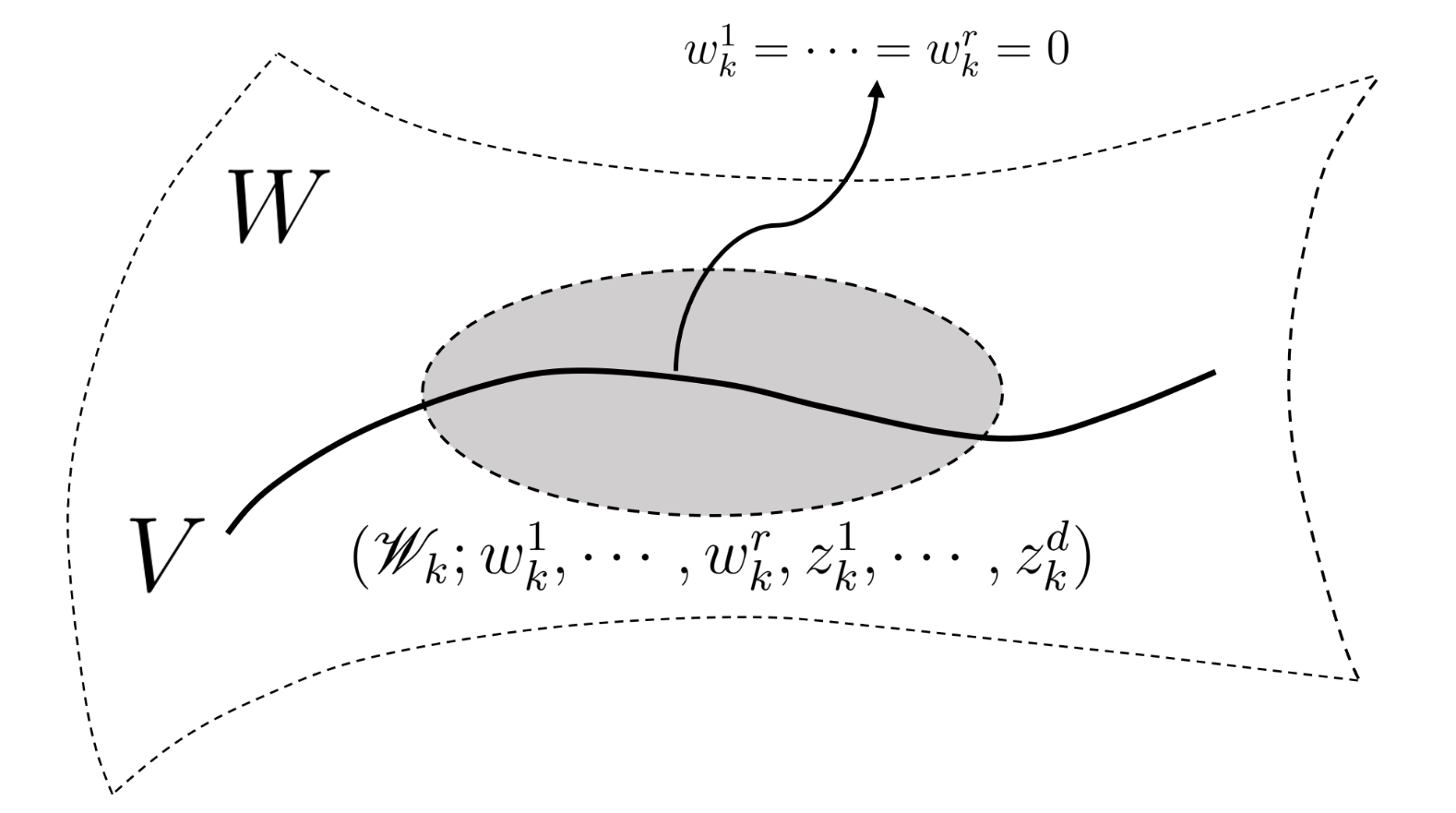}
    \caption{System of canonical local charts}
\end{figure}

\begin{definition}
\label{normal bundle}
The normal bundle $NV$ over a compact complex submanifold $V$ of $W$ is defined by the following short exact sequence:
\begin{equation*}
0 \rightarrow TV \rightarrow TW|_V \rightarrow NV \rightarrow 0.
\end{equation*}
\end{definition}

Consider $\Big\{\big(\mathscr{W}_k, w^1_k, \cdots, w_k^r, z^1_k, \cdots, z^d_k\big)\Big\}_{k \in I}$ as the system of canonical local charts in the neighborhood of $V$, as defined in Definition \ref{system of canonical local charts}. Let $\mathscr{V}_k := \mathscr{W}_k \cap V$ denote the intersection of $\mathscr{W}_k$ and $V$. The collection $\big\{\mathscr{V}_k\big\}_{k \in I}$ forms local charts for $V$, each with coordinates $\big(z_k^1, \cdots, z_k^d\big)$. To elucidate the structure of the normal bundle $NV$, we consider the open covering $V = \bigcup_{k \in I} \mathscr{V}_k$ and define the transformation functions between overlapping charts $\mathscr{W}_i$ and $\mathscr{W}_k$ as follows:
\[
\begin{cases}
w^\lambda_i = f_{ik}^\lambda(w_k, z_k) & \text{for } \lambda = 1, 2, \cdots, r \\
z^\mu_i = g_{ik}^\mu(w_k, z_k) & \text{for } \mu = 1, 2, \cdots, d
\end{cases}
\]
The normal bundle $NV$ is then constructed as $NV = \bigcup_{k \in I} (\mathscr{V}_k \times \mathbb{C}^r) / \sim$, where $\sim$ denotes the equivalence relation defined by the system of transition matrices $\big\{B_{ik}(z_k)\big\}$:
\[
B_{ik}(z_k) = \left( \left. \frac{\partial w_i^\lambda}{\partial w_k^\mu} \right|_{\mathscr{V}_i \cap \mathscr{V}_k} \right)_{\lambda, \mu = 1, \cdots, r} = \left( \left. \frac{\partial f_{ik}^\lambda(w_k, z_k)}{\partial w_k^\mu} \right|_{w_k = 0} \right)_{\lambda, \mu = 1, \cdots, r}
\]

Having established the local structure and normal bundle of a single compact complex submanifold $V$ in $W$, we now extend our discussion to a broader context. We consider an analytic family of compact submanifolds, represented by the map $\pi: \mathcal{V} \rightarrow N$. In this expanded framework, we aim to investigate the overall structure of each complex submanifold within this analytic family, as well as the corresponding structures of their normal bundles.

Let $\pi: \mathcal{V} \rightarrow N$ denote an analytic family of compact submanifolds within the complex manifold $W$. Assume $\mathbb{B}_\varepsilon^l = \Big\{t \in \mathbb{C}^l \mid \|t\|_\infty < \varepsilon\Big\}$ to be a sufficiently small neighborhood around the origin in $\mathbb{C}^l$, which will be considered as an open chart in $N$. This setting allows us to view $\pi: \mathcal{V} \rightarrow \mathbb{B}_\varepsilon^l$ as a parameterized analytic family of compact submanifolds of $W$ with dimension $d$. In this framework, $V_t$ is treated as a submanifold of both $W$ and the product space $W \times \{t\}$.

Given $\varepsilon > 0$, there exists a finite open covering $V_0 := \pi^{-1}(0) \subset \bigcup\limits_{k \in I} \mathscr{U}_k$ in $W \times \mathbb{B}_\varepsilon^l$. Each set $\mathscr{U}_k$ has local coordinates $\big(\mathscr{U}_k, w^1_k, \cdots, w^r_k, z^1_k, \cdots, z^d_k, t^1, \cdots, t^l\big)$ satisfying:

\begin{enumerate}
\item $\mathscr{U}_k = \mathscr{W}_k \times \mathbb{B}^l_\varepsilon$ for an open set $\mathscr{W}_k$ in $W$.
\item $\Big(\mathscr{W}_k, w^1_k, \cdots, w^r_k, z^1_k, \cdots, z^d_k\Big)$ forms a local chart on $W$.
\item For all $t \in \mathbb{B}_\varepsilon^l$, $\Big(\mathscr{V}_{k,t} := \mathscr{U}_k \cap V_t, z^1_k, \cdots, z^d_k\Big)$ serves as a local chart on $V_t$.
\end{enumerate}
The intersection of $V_t$ and $\mathscr{U}_k$ is represented by the simultaneous equations:
\[
\begin{cases}
w_k^1 = \varphi_k^1(z_k^1, \cdots, z_k^d, t^1, \cdots, t^l)\\
\vdots\\
w_k^r = \varphi_k^r(z_k^1, \cdots, z_k^d, t^1, \cdots, t^l)
\end{cases}
\]
where $\varphi^\lambda_k$ are holomorphic functions of $z_k$ and $t$. Furthermore, within $\mathscr{U}_k$, $V_0$ is defined by $w^1_k = \cdots = w^r_k = 0$.

\begin{figure}[htbp]
    \centering
    \includegraphics[width=0.8\textwidth]{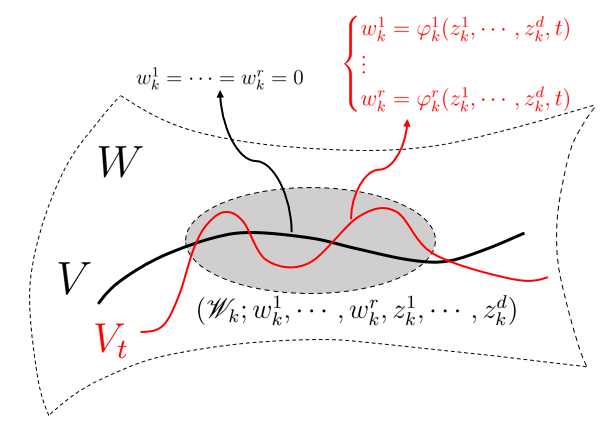}
    \caption{Deformation of compact complex submanifolds}
\end{figure}

In this scenario, the collection of open charts $\bigcup_{k \in I} (\mathscr{U}_k; w_k, z_k)$ forms a set of canonical local charts on $W$ for the neighborhood of $V$, as previously defined in Definition \ref{normal bundle}. We introduce new local coordinates on $\mathscr{W}_k$ by defining:
\[
\tilde w^\lambda_{k,t} := w_k^\lambda - \varphi^\lambda_k(z_k, t) \quad \text{and} \quad \tilde z_k^\mu = z_k^\mu,
\]
where $V_t$ is represented by $\tilde w^1_{k,t} = \cdots = \tilde w^r_{k,t} = 0$ within $\mathscr{U}_k$. Accordingly, the collection of open charts $\bigcup_{k \in I} (\mathscr{U}_k; \tilde z_k, \tilde w_k)$ then constitutes a set of canonical local charts for $W$ in the neighborhood of each fiber $V_t$, as per Definition \ref{normal bundle}.

Following the same approach as before, the transformation functions between overlapping charts $\mathscr{W}_i$ and $\mathscr{W}_k$ are given by:
\[
\begin{cases}
w^\lambda_i = f_{ik}^\lambda(w_k, z_k) & \text{for } \lambda = 1, 2, \dots, r, \\
z^\mu_i = g_{ik}^\mu(w_k, z_k) & \text{for } \mu = 1, 2, \dots, d,
\end{cases}
\]
where $w_k = (w_k^1, \dots, w_k^r)$ and $z_k = (z_k^1, \dots, z_k^d)$. Subsequently, we define the normal bundle $NV_t$ of each fiber $V_t$ in the analytic family $\pi: \mathcal{V} \rightarrow \mathbb{B}_\varepsilon^l$ as follows: 

\begin{definition}
\label{Normal bundle}
Define the normal bundle of $V_t$ as a holomorphic vector bundle $B_t$ of rank $r$ over $V_t$. This vector bundle is characterized by a set of bundle transition functions $B_{t,ik}$ associated with the open chart covering $V_t = \bigcup_{k \in I} \mathscr{V}_{k,t}$. The bundle transition functions $B_{t,ik}(z_k)$ are specified by:
\[
\left(\left.\frac{\partial \tilde w^\lambda_i}{\partial \tilde w_k^\mu}\right|_{\mathscr{V}_{i,t} \cap \mathscr{V}_{k,t}}\right)_{\lambda, \mu = 1, \dots, r} = \left(\left.\frac{\partial f_{ik}^\lambda(w_k, z_k)}{\partial w^\mu_k} - \frac{\partial \varphi_i^\lambda(z_k, t)}{\partial z_i^\alpha} \cdot \frac{\partial g^\alpha_{ik}(w_k, z_k)}{\partial w^\mu_k}\right|_{\mathscr{V}_{i,t} \cap \mathscr{V}_{k,t}}\right)_{\lambda, \mu = 1, \dots, r}.
\]
\end{definition}

Having established the normal bundle $B_t$ of each fiber $V_t$ in the family $\mathcal{V}$, we now proceed to generalize this concept to the entire family. This allows us to consider the variation of these normal bundles across compact complex submanifold in $\mathcal{V}$, leading to the definition of a holistic structure that encompasses whole $\mathcal{V}$.

\begin{definition}
Define the holomorphic function $B_{ik} \colon \mathcal{V} \cap \mathscr{U}_i \cap \mathscr{U}_k \rightarrow GL(r, \mathbb{C})$ such that $B_{ik}|_{V_t} = B_{t,ik}$. The holomorphic vector bundle $\mathcal{B}$ over $\mathcal{V}$ is then defined by the collection of these bundle transition functions $B_{ik}$, corresponding to the open chart covering $\mathcal{V} = \bigcup_{k \in I} (\mathscr{U}_i \cap \mathcal{V})$. The projection $\omega \colon \mathcal{B} \rightarrow \mathcal{V}$ denotes the \textit{complex analytic family of normal bundles on $\mathcal{V}$}.
\end{definition}

The family of normal bundles $\omega \colon \mathcal{B} \rightarrow \mathcal{V}$ constitutes a complex analytic family over the mapping $\pi \colon \mathcal{V} \rightarrow \mathbb{B}_\varepsilon^l$, as defined in Definition \ref{differentiable family of pi}. Furthermore, according to Theorem \ref{Principle of Upper Semi-continuity}, for any $t \in \mathbb{B}^l_{\varepsilon}$ sufficiently close to 0 and for any integer $q \geq 0$, the dimension of the cohomology group $H^q\left(V_{t}, B_{t}\right)$ is less than or equal to that of $H^q\left(V_0, B_0\right)$.

Now that we have described the normal bundles $NV_t$ associated with each fiber $V_t$ of the analytic family $\pi \colon \mathcal{V} \rightarrow \mathbb{B}_\varepsilon^l$, it is pertinent to delve into the mechanisms of their deformation. This leads us to explore the concept of infinitesimal deformation in the context of complex manifolds, a notion central to understanding the changes in compact complex submanifolds under small perturbations. Kodaira and Spencer's pioneering work \cite{kodaira1958deformations} provides a foundation for this exploration. For every fiber $V_t$ of $\pi \colon \mathcal{V} \rightarrow \mathbb{B}_\varepsilon^l$ and any tangent vector $\frac{\partial}{\partial t} := \gamma_\rho \frac{\partial}{\partial t^\rho}$ at $t$ in $N$, let $B_t$ denote the normal bundle of $V_t$, as defined in Definition \ref{Normal bundle}. According to Kodaira \cite{kodaira1962theorem}, the set of holomorphic functions $\{\Psi_{k,t} \colon \mathscr{V}_{k,t} \rightarrow \mathbb{C}^r\}_{k \in I}$, defined by 
\[
\Psi_{k,t}(z_k) = \left(\left.\frac{\partial \varphi^\lambda_k(z_k,t)}{\partial t}\right|_{\mathscr{V}_{k,t}}\right)_{\lambda = 1, \cdots, r},
\]
satisfies the equation $\Psi_{i,t}(z_i) = B_{t,ik} \cdot \Psi_{k,t}(z_k)$. Thus, $\{\Psi_{k,t}(z_k)\}_{k \in I}$ induces a holomorphic section of $B_t$ over the open cover $V_t = \bigcup_{k \in I} \mathscr{V}_{k,t}$.

\begin{definition}
The holomorphic section of $B_t$ induced by $\{\Psi_{k,t}(z_k)\}_{k \in I}$ is referred to as the \textit{infinitesimal deformation} of $V_t$ along the tangent vector $\frac{\partial}{\partial t} = \gamma_\rho \frac{\partial}{\partial t^\rho}$ at $t \in N$. The linear mapping 
\[
\sigma_t \colon T_tN \rightarrow H^0(V_t, B_t) \quad \text{given by} \quad \frac{\partial}{\partial t} \mapsto \{\Psi_{k,t}\}_{k \in I}
\]
is known as the \textit{Kodaira-Spencer map}.
\end{definition}

We now focus on the specific properties and deformations of a compact complex submanifold within such a manifold. In the context of a complex manifold $W$ of dimension $r + d$, consider $V$, a compact complex submanifold of dimension $d$. Utilizing the canonical local charts $(\mathscr{W}_k, w^1_k, \cdots, w_k^r, z^1_k, \cdots, z^d_k)_{k \in I}$ near $V$ as per Definition \ref{system of canonical local charts}, we denote $\mathscr{V}_k$ as the intersection $\mathscr{W}_k \cap V$. Assuming $\dim_\mathbb{C} H^0(V, NV) = l$ with a basis $\{\beta_1, \cdots, \beta_l\}$ for $H^0(V, NV)$, each $\beta_\rho$ in $\mathscr{V}_k \times \mathbb{C}^r$ is represented by 
\[
\beta_{\rho,k} = \Big(\beta_{\rho,k}^1, \cdots, \beta_{\rho,k}^r\Big)
\]
with $\beta^\lambda_{\rho,k}(z_k^1, \cdots, z_k^d)$ as holomorphic functions on $\mathscr{V}_k$. Kodaira's Theorem \cite[Theorem 1]{kodaira1962theorem} demonstrates the surjectivity of the Kodaira-Spencer map for $V$ under specific conditions on $NV$.

\begin{prop} 
\label{canonical local deformation}
Suppose $H^1(V, NV) = 0$. For a sufficiently small positive number $\varepsilon$, there exist holomorphic vector functions $\varphi_k \colon \mathscr{V}_k \times \mathbb{B}^l_\varepsilon \rightarrow \mathbb{C}^r$, defined as $\varphi_k(z_k, t) = (\varphi^1_k, \cdots, \varphi^r_k)$, satisfying these transition and boundary conditions:

Transition Conditions:
\[
\begin{cases}
\varphi_i^\lambda = f_{ik}^\lambda(\varphi^1_k, \cdots, \varphi^r_k, z^1_k, \cdots, z^d_k) & \quad \lambda = 1, \cdots, r, \\
z_i^{\mu} = g_{ik}^\mu(\varphi^1_k, \cdots, \varphi^r_k, z^1_k, \cdots, z^d_k) & \quad \mu = 1, \cdots, d,
\end{cases}
\]

Boundary Conditions:
\[
\begin{cases}
\varphi_k^\lambda(z^1_k, \cdots, z_k^d, 0) = 0 & \quad \lambda = 1, \cdots, r, \\
\left.\frac{\partial \varphi^\lambda_k}{\partial t^\rho}\right|_{t = 0} = \beta^\lambda_{\rho,k} & \quad \rho = 1, \cdots, l.
\end{cases}
\]
\end{prop}

\begin{remark}
The boundary condition implies that the function system
\[
\left\{\left(\left.\frac{\partial \varphi_k^\lambda}{\partial t^\rho}\right|_{t=0}\right)_{\lambda=1,\cdots,r} \colon \mathscr{V}_k \rightarrow \mathbb{C}^r\right\}_{k \in I}
\]
induces a holomorphic section of $H^0(V, NV)$, exactly corresponding to $\beta_\rho$ for each $\rho = 1, \cdots, l$.
\end{remark}

Let $V_t$ be defined near $V$ as 
\[
V_t \cap \mathscr{W}_k = \Big\{w_k^\lambda = \varphi^\lambda_k(z_k) \mid \lambda = 1, \cdots, r\Big\}
\]
in $\mathscr{W}_k$, where $\varphi_k$ are the holomorphic functions defined in Definition \ref{Normal bundle}. The transition condition ensures that $V_t$ forms a submanifold of $W$ for each $t \in \mathbb{B}_\varepsilon^l$. Defining $\mathcal{V} = \bigsqcup_{t \in \mathbb{B}^l_\varepsilon} V_t$ as a submanifold of $W \times \mathbb{B}^l_\varepsilon$, we obtain an analytic family $\pi^V \colon \mathcal{V} \rightarrow \mathbb{B}_\varepsilon^l$, consisting of compact submanifolds of $W$. This family has the properties: $V_0 = V$, $B_0 = NV$, and the Kodaira-Spencer map at $t = 0$ is an isomorphism, thereby constituting the \textit{canonical local deformation of $V$}.

\begin{figure}[htbp]
    \centering
    \includegraphics[width=0.8\textwidth]{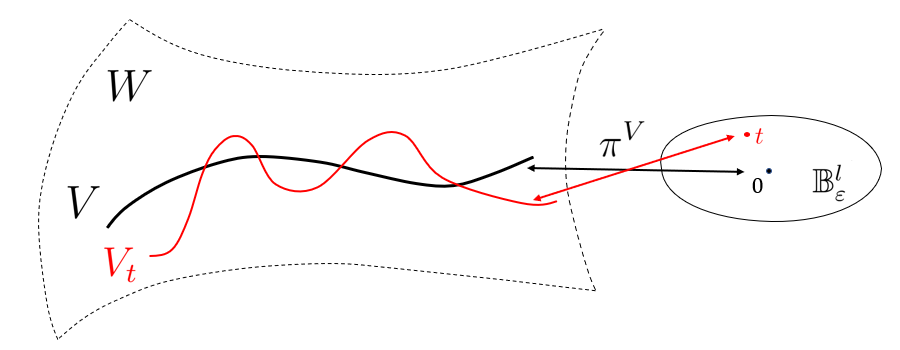}
    \caption{Canonical local deformation}
\end{figure}

Building upon the foundational concepts we have discussed, Kodaira further extends these ideas in \cite{kodaira1962theorem} by presenting a more comprehensive theorem. This theorem not only encompasses the earlier results but also provides deeper insights into the structure of analytic families of compact complex submanifolds.

\begin{theorem}
\label{main Kod1}
If $H^1(V, NV) = 0$ and $\dim_\mathbb{C}H^0(V, NV)=l$, then there exists an analytic family $\pi \colon \mathcal{V} \rightarrow \mathbb{B}_\varepsilon^l$ of compact complex submanifolds of $W$ such that $V_0 = V$. This family $\pi \colon \mathcal{V} \rightarrow \mathbb{B}_\varepsilon^l$ satisfies the property that the Kodaira-Spencer map $\sigma_t \colon T_t\mathbb{B}_\varepsilon^l \rightarrow H^0(V, B_t)$ is an isomorphism for all $t \in \mathbb{B}^l_\varepsilon$. Additionally, at each point $t \in \mathbb{B}^l_\varepsilon$, the family is maximal.
\end{theorem} 

The notion of 'maximal' in this context is further elaborated in the following definition.

\begin{definition}
\label{maximal}
An analytic family $\pi \colon \mathcal{V} \rightarrow N$ of compact submanifolds $V_t$, $t \in N$, of $W$ is said to be maximal at a point $t_0$ of $N$ if and only if, for any analytic family $\pi^\prime \colon \mathcal{V}^\prime \rightarrow N^\prime$ of compact submanifolds $V^\prime_s$, $s \in N^\prime$, of $W$ such that $V^\prime_{s_0} = V_{t_0}$ for a point $s_0 \in N$, there exists a neighborhood $U^\prime$ of $s_0$ in $N^\prime$ and a holomorphic map $h \colon U^\prime \rightarrow N$ sending $s_0$ into $t_0$ such that $V^\prime_s = V_{h(s)}$ for all $s \in U^\prime$, where $V^\prime_s = V_t$ indicates that $V^\prime_s$ and $V_t$ are the same submanifold of $W$.
\end{definition}

\begin{remark}
It is crucial to emphasize that the analytic family $\pi \colon \mathcal{V} \rightarrow \mathbb{B}^l_\varepsilon$, corresponding to the canonical local deformation of $V$ as defined in Proposition \ref{canonical local deformation}, fully satisfies the conditions stated in Theorem \ref{main Kod1} when $\varepsilon$ is sufficiently small.
\end{remark}

Let $X$ be a compact complex submanifold of $W$ such that $H^1(X, NX) = 0$. In line with the conditions outlined in Theorem \ref{main Kod1}, we utilize the concept of canonical local deformation as defined in Proposition \ref{canonical local deformation}. This approach facilitates the exploration of deformations of $X$ within the manifold $W$, providing a framework to understand how $X$ can be varied or transformed while remaining within $W$. To formalize this concept, we introduce the notion of a regular deformation space:

\begin{definition}
\label{regular deformation space}
Let $X$ be a compact submanifold of $W$ of dimension $d$. A compact complex submanifold $V$ of $W$ is called a \textit{deformation} of $X$ in $W$ if there exist analytic families $\pi_k \colon \mathcal{V}_k \to N_k$, $k = 0, \cdots, n-1$, with the following properties: $X = \pi_0^{-1}(s_0)$, $\pi_{k+1}^{-1}(s_{k+1}) = \pi_k^{-1}(t_k)$, and $V = \pi_{n-1}^{-1}(t_{n-1})$ for some $s_k, t_k \in N_k$. The set of all such deformations of $X$ in $W$ is denoted as $\mathscr{M}_X$.

Assuming $H^1(X, NX) = 0$, the \textit{regular deformation space} of $X$ in $W$, denoted as $\widetilde{\mathscr{M}_X}$, is defined as the subset of $\mathscr{M}_X$ consisting of those deformations $V$ satisfying $H^1(V, NV) = 0$.
\[
\widetilde{\mathscr{M}_X} = \Big\{ V \in \mathscr{M}_X \Big| H^1(V, NV) = 0 \Big\}.
\]
\end{definition}

\begin{figure}[htbp]
    \centering
    \includegraphics[width=0.7\textwidth]{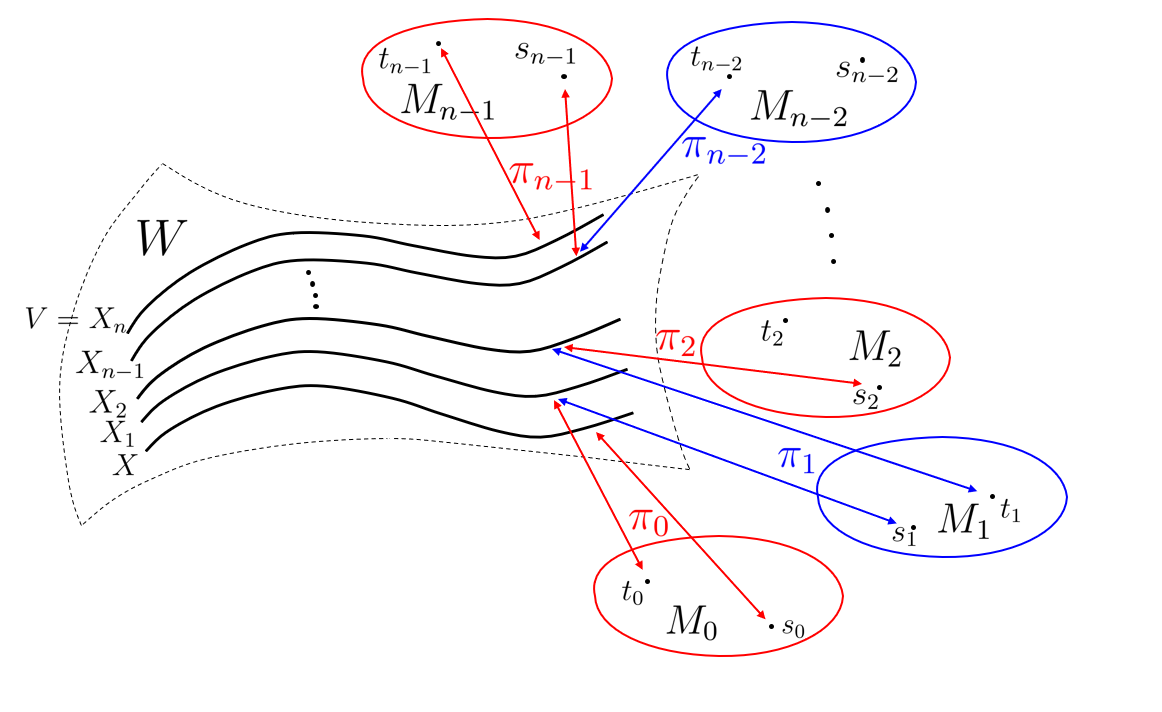}
    \caption{Deformation space}
\end{figure}

\begin{lemma}
\label{topo}
Let $X$ be a compact complex submanifold of $W$ and $H^1(X, NX) = 0$. Then $\widetilde{\mathscr{M}_X}$ has a $C_1$ and $T_2$ topology.
\end{lemma}

\begin{proof} 
In the regular deformation space $\widetilde{\mathscr{M}_X}$, each manifold $V$ is associated with an analytic family $\pi^V \colon \mathcal{V} \rightarrow \mathbb{B}^l_\varepsilon$, comprising compact submanifolds of a fixed manifold $W$. For a suitably small $\varepsilon$, $\pi^V$ adheres to the criteria:
\begin{enumerate}
    \item $V_0$, the fiber over the origin in $\mathbb{B}^l_\varepsilon$, equals $V$.
    \item For every $t \in \mathbb{B}^l_\varepsilon$, the Kodaira-Spencer map $\sigma_t$ related to $\pi^V$ is an isomorphism.
    \item $\pi^V$ is maximal at each point $t \in \mathbb{B}^l_\varepsilon$ (refer to Definition \ref{maximal}).
\end{enumerate}
This family stems from $V$'s canonical local deformation. Let $l := \dim H^0(V, NV)$ and $V_t := \left(\pi^V\right)^{-1}(t)$ represent the fiber for each $t \in \mathbb{B}^l_\varepsilon$. With a sufficiently minimized $\varepsilon$, we deduce:
\[
\dim H^1(V_t, NV_t) \leq \dim H^1(V, NV) = 0,
\]
in line with the Principle of Upper Semi-continuity (refer to Theorem \ref{Principle of Upper Semi-continuity}). This connotes $\dim H^1(V_t, NV_t) = 0$, confirming that $V_t$ is included in $\widetilde{\mathscr{M}_X}$.

Let $U_{\pi^V}\left(\frac{1}{n}\right) = \left\{ V_t = (\pi^V)^{-1}(t) \mid t \in \mathbb{B}^l_{\frac{1}{n}} \right\}$ denote neighborhoods of $V$ in $\widetilde{\mathscr{M}_X}$, where $n$ is chosen to be sufficiently large. Consider any $Y, Z \in \widetilde{\mathscr{M}_X}$ such that the intersection $U_{\pi^Y}\left(\frac{1}{n_Y}\right) \cap U_{\pi^Z}\left(\frac{1}{n_Z}\right)$ is non-empty. For a manifold $V$ in this intersection, i.e., $V \in U_{\pi^Y}\left(\frac{1}{n_Y}\right) \cap U_{\pi^Z}\left(\frac{1}{n_Z}\right)$, it follows that $U_{\pi^V}(\varepsilon) \subset U_{\pi^Y}\left(\frac{1}{n_Y}\right) \cap U_{\pi^Z}\left(\frac{1}{n_Z}\right)$ for a sufficiently small $\varepsilon$. This inclusion holds because both $\pi^Y$ and $\pi^Z$ are maximal at each point in their domains. It is readily observed that
\[
\bigcup_{V \in \widetilde{\mathscr{M}_X}, n \gg 1} U_{\pi^V}\left(\frac{1}{n}\right) = \widetilde{\mathscr{M}_X}.
\]
Therefore, the collection $\left\{U_{\pi^V}\left(\frac{1}{n}\right)\right\}_{n \gg 1}$ forms a countable basis for the topology at each point $V$ in $\widetilde{\mathscr{M}_X}$, suggesting that $\widetilde{\mathscr{M}_X}$ is endowed with a $C_1$ topology. 

To establish the $T_2$ (Hausdorff) property in the regular deformation space $\widetilde{\mathscr{M}_X}$, a more detailed examination of the canonical local deformations is necessary. Consider distinct manifolds $Y$ and $Z$ in $\widetilde{\mathscr{M}_X}$, meaning that $Y$ and $Z$ are separate submanifolds of $W$. Since both $Y$ and $Z$ are compact submanifolds, there exist points $y \in Y \setminus Z$ and $z \in Z \setminus Y$ in $W$. We select canonical local charts for $W$ around $Y$ and $Z$ as follows:
\begin{align*}
&\Big\{\big(\mathscr{W}_{Y,k}; w_{Y,k}^1, \ldots, w_{Y,k}^r, z_{Y,k}^1, \ldots, z_{Y,k}^d\big)\Big\}_{k \in I},\\
&\Big\{\big(\mathscr{W}_{Z,k}; w_{Z,k}^1, \ldots, w_{Z,k}^r, z_{Z,k}^1, \ldots, z_{Z,k}^d\big)\Big\}_{k \in I}.
\end{align*}
The corresponding restricted local holomorphic charts on $Y$ and $Z$ are 
\[
\Big\{\big(\mathscr{Y}_k := \mathscr{W}_{Y,k}; z_{Y,k}^1, \ldots, z_{Y,k}^d\big)\Big\}_{k \in I},\quad \Big\{\big(\mathscr{Z}_k := \mathscr{W}_{Z,k}; z_{Z,k}^1, \ldots, z_{Z,k}^d\big)\Big\}_{k \in I}.
\]
The canonical local deformations are defined by holomorphic function systems $\varphi_k:\mathscr{Y}_k \times \mathbb{B}^l_{\varepsilon_Y} \to \mathbb{C}^r$ and $\psi_k:\mathscr{Z}_k \times \mathbb{B}^{l^\prime}_{\varepsilon_Z} \to \mathbb{C}^r$, such that
\[
Y_t := \pi^{-1}(t) = \{ w_{Y,k}^\lambda = \varphi_k^\lambda(z_{Y,k}, t) \}, \quad Z_t := \pi^{-1}(t) = \{ w_{Z,k}^\lambda = \psi_k^\lambda(z_{Z,k}, t) \}
\]
for each $k \in I$ and $\lambda = 1, \ldots, r$. Here, $l = \operatorname{Dim} H^0(Y, NY)$ and $l^\prime = \operatorname{Dim} H^0(Z, NZ)$.

Assuming $\mathscr{W}_{Y,0}$ and $\mathscr{W}_{Z,0}$ are centered at $y$ and $z$, respectively, we can shrink each $\mathscr{W}_{Y,k}$ and $\mathscr{W}_{Z,k}$, and adjust $\varepsilon_Y$ and $\varepsilon_Z$ as needed. Without loss of generality, assume 
\[
\mathscr{W}_{Y,0} \cap \left( \bigcup_{k \in I} \mathscr{W}_{Z,k} \right) = \mathscr{W}_{Z,0} \cap \left( \bigcup_{k \in I} \mathscr{W}_{Y,k} \right) = \emptyset,
\]
in accordance with $W$ being a $T_3$ topology space. For any $t \in \mathbb{B}^l_{\varepsilon_Y}$ and $t^\prime \in \mathbb{B}^{l^\prime}_{\varepsilon_Z}$, we have $Y_t := \left(\pi^Y\right)^{-1}(t) \neq Z_{t^\prime} := \left(\pi^Z\right)^{-1}(t^\prime)$, demonstrated by
\[
\begin{cases}
Y_t \cap \mathscr{W}_{Y,0} \neq \emptyset, & Y_t \cap \mathscr{W}_{Z,0} = \emptyset, \\
Z_{t^\prime} \cap \mathscr{W}_{Y,0} = \emptyset, & Z_{t^\prime} \cap \mathscr{W}_{Z,0} \neq \emptyset.
\end{cases}
\]
Hence, the open neighborhoods $U_{\pi^Y}(\varepsilon_Y)$ and $U_{\pi^Z}(\varepsilon_Z)$ of $Y$ and $Z$ within $\widetilde{\mathscr{M}_X}$ are disjoint.
\end{proof}

The topology previously defined on the space $\widetilde{\mathscr{M}_X}$ is hereinafter denoted as $\mathscr{S}_X$. We consider the topological space $\left(\widetilde{\mathscr{M}_X}, \mathscr{S}_X\right)$ and assume that it satisfies the $C_2$ property. Furthermore, we define $\widetilde{\mathscr{M}^l_X}$ as the subset of $\widetilde{\mathscr{M}_X}$ consisting of elements $V$ for which $\operatorname{Dim} H^0(V,NV)=l$. It is noteworthy that $\widetilde{\mathscr{M}^l_X}$ constitutes a open subset in the space $\left(\widetilde{\mathscr{M}_X}, \mathscr{S}_X\right)$.

\begin{theorem}
\label{complex regular deformation space}
Given that $\widetilde{\mathscr{M}^l_X}$ is non-empty, the pair $(\widetilde{\mathscr{M}^l_X}, \mathscr{S}_X)$ manifests as a complex manifold with dimension $l$, for each $l \in \mathbb{Z}_{\geq0}$.
\end{theorem}

\begin{proof}
Expanding on the notations introduced in the preceding lemma, consider for each $V \in \widetilde{\mathscr{M}^l_X}$, the map
\[
\pi^V: U_{\pi^V}(\varepsilon) \rightarrow \mathbb{B}_\varepsilon^l, \quad \text{where} \quad V_t = \left(\pi^V\right)^{-1}(t) \mapsto t \in \mathbb{B}_\varepsilon^l.
\]
This map $\pi^V$ manifests as a local homeomorphism for sufficiently small $\varepsilon$, justified by the isomorphism of the Kodaira-Spencer map. Consider any two elements $Y, Z \in \widetilde{\mathscr{M}_X^l}$, such that $V$ belongs to both $U_{\pi^Y}(\varepsilon_Y)$ and $U_{\pi^Z}(\varepsilon_Z)$. Choose $\varepsilon$ small enough to ensure $U_{\pi^V}(\varepsilon) \subseteq U_{\pi^Y}(\varepsilon_Y) \cap U_{\pi^Z}(\varepsilon_Z)$.

To establish the lemma, it suffices to demonstrate that $\pi^Y \circ \left(\pi^Z\right)^{-1}$ is holomorphic at the point $\left(\pi^Z\right)^{-1}(V)$.

Observe that the family $\pi^Y: U_{\pi^Y}(\varepsilon_Y) \rightarrow \mathbb{B}^l_{\varepsilon_Y}$ is maximal at the point $\pi^Y(V)$. Consequently, for a sufficiently small $\varepsilon$, the family $\pi^V: U_{\pi^V}(\varepsilon) \rightarrow \mathbb{B}^l_{\varepsilon}$ is subsumed within $\pi^Y$. Hence, there exists a holomorphic function $h: \mathbb{B}^l_\varepsilon \rightarrow \mathbb{B}^l_{\varepsilon_Y}$ satisfying $\left(\left(\pi^V\right)^{-1}(s)\right) = \left(\left(\pi^Y\right)^{-1}(h(s))\right)$. This implies that $\pi^Y \circ \left(\left(\pi^V\right)^{-1}(s)\right) = h(s)$ defines a biholomorphic map in the vicinity of $s=0$. A similar argument establishes that $\pi^Z \circ \left(\left(\pi^V\right)^{-1}(s)\right)$ is biholomorphic near $s=0$. Therefore,
\[
\pi^Y \circ \left(\pi^Z\right)^{-1} = \left(\pi^Y \circ \left(\left(\pi^V\right)^{-1}\right)\right) \circ \left(\pi^Z \circ \left(\left(\pi^V\right)^{-1}\right)\right)^{-1},
\]
is holomorphic at the point $\left(\left(\pi^Z\right)^{-1}(V)\right)$.
\end{proof}


\section{real structures on $\bigoplus^{2n}\mathcal{O}(1)\rightarrow\mathbb{CP}^1$}

In this section, we delve into the Riemann sphere $\mathbb{CP}^1$ and the vector bundle $\bigoplus^{2n}\mathcal{O}(1)$ over $\mathbb{CP}^1$. Additionally, we explore $\overline{\mathbb{CP}^1}$, a smooth manifold that is topologically identical to $\mathbb{CP}^1$ but is equipped with a different system of complex open charts. This distinction is particularly significant in the context of complex geometry and analysis. In this setting, the antipodal map $\sigma$ serves as a biholomorphic map from $\mathbb{CP}^1$ to $\overline{\mathbb{CP}^1}$. Utilizing $\sigma$, we define real structures on the vector bundle $\bigoplus^{2n}\mathcal{O}(1)$, corresponding to a quaternionic structure on $H^0(\mathbb{CP}^1, (\bigoplus^{2n}\mathcal{O}(1))\otimes\mathcal{O}(-1)) \cong \mathbb{C}^{2n}$. These structures are fundamental for the study of hyperkähler manifolds, a topic we will explore in Section 4. Our analysis and calculations throughout this section will be grounded in the use of the standard open cover of $\mathbb{CP}^1$, ensuring that our approach is both systematic and comprehensive.

The Riemann sphere can be depicted as the complex projective line $\mathbb{CP}^1 = \big\{ [z_0:z_1] \big| (z_0, z_1) \in \mathbb{C}^2 \setminus {0} \big\}$. For analytical convenience, we adopt a standard open covering of $\mathbb{CP}^1$, denoted as $U_0 \cup U_1$, where $U_0 = \big\{ [z_0:z_1] \in \mathbb{CP}^1 \big| z_0 \neq 0 \big\}$ and $U_1 = \big\{ [z_0:z_1] \in \mathbb{CP}^1 \big| z_1 \neq 0 \big\}$. The local holomorphic coordinates are defined as $\zeta_0 \big( [z_0:z_1] \big) = \frac{z_1}{z_0}$ on $U_0$ and $\zeta_1 \big( [z_0:z_1] \big) = \frac{z_0}{z_1}$ on $U_1$. Notably, $\zeta_0 \in U_0$ corresponds bijectively to $\zeta_1 \in U_1$ if and only if $\zeta_0 \cdot \zeta_1 = 1$. Furthermore, we define $\overline{\mathbb{CP}^1}$ to possess the same smooth manifold structure as $\mathbb{CP}^1$, with open charts $V_i := U_i$ for $i = 0,1$. Local holomorphic coordinates are introduced as $\eta_0\big([z_0:z_1]\big) = \frac{\bar{z}_1}{\bar{z}_0}$ on $V_0$, and $\eta_1\big([z_0:z_1]\big) = \frac{\bar{z}_0}{\bar{z}_1}$ on $V_1$. The conditions for equivalence between $\eta_0 \in V_0$ and $\eta_1 \in V_1$ are met if and only if $\eta_0 \cdot \eta_1 = 1$. Consequently, these holomorphic charts $(V_i, \eta_i, i = 0,1)$ endow $\overline{\mathbb{CP}^1}$ with a complex manifold structure.

In the framework of smooth manifolds, the identity map, denoted by $\mathfrak{I}$, maps the complex projective line $\mathbb{CP}^1$ to its conjugate $\overline{\mathbb{CP}^1}$. This map is characterized as antiholomorphic and can be articulated in the following manner:
\[
\begin{cases}
\mathfrak{I}_0 : U_0 \rightarrow V_0, \quad \zeta_0 \mapsto \bar{\zeta}_0, \\
\mathfrak{I}_1 : U_1 \rightarrow V_1, \quad \zeta_1 \mapsto \bar{\zeta}_1.
\end{cases}
\]
Likewise, the inverse map $\mathfrak{I}^{-1} : \overline{\mathbb{CP}^1} \rightarrow \mathbb{CP}^1$ retains the property of being antiholomorphic.

In a broader context, consider $(M, \mathbf{I})$ as a complex manifold. The conjugate manifold, denoted $\overline{M}$, is defined to be a smooth manifold that shares the same smooth structure as $M$, but with the complex structure altered to $-\mathbf{I}$. This definition implies that the identity map from the complex manifold $M=(M, \mathbf{I})$ to its conjugate $\overline{M}=(M, -\mathbf{I})$ is an antiholomorphic map.

It is well established that there exists a unique complex structure on the 2-sphere $S^2$, implying the existence of a biholomorphic map from $\mathbb{CP}^1$ to $\overline{\mathbb{CP}^1}$. A notable instance of such a map is the antipodal map on $S^2$, denoted by $\sigma$. This map can be delineated by the following system of equations:
\[
\begin{cases}
\sigma_0 : U_0 \rightarrow V_1, \quad \zeta_0 \mapsto -\zeta_0, \\
\sigma_1 : U_1 \rightarrow V_0, \quad \zeta_1 \mapsto -\zeta_1.
\end{cases}
\]
Define $\bar{\sigma} := \mathfrak{I}^{-1} \circ \sigma \circ \mathfrak{I}^{-1}$. It is straightforward to verify that $\sigma \circ \bar{\sigma} = \text{id}_{\overline{\mathbb{CP}^1}}$ and $\bar{\sigma} \circ \sigma = \text{id}_{\mathbb{CP}^1}$.

Let $\bigoplus^{2n}\mathcal{O}(1) = (U_0 \times \mathbb{C}^{2n}) \cup (U_1 \times \mathbb{C}^{2n})/\sim$ be a holomorphic vector bundle over $\mathbb{CP}^1$. The local coordinates on $U_0 \times \mathbb{C}^{2n}$ and $U_1 \times \mathbb{C}^{2n}$ are denoted as $(\zeta_0, x_0^1, \ldots, x_0^{2n})$ and $(\zeta_1, x_1^1, \ldots, x_1^{2n})$, respectively. In these coordinates, $(\zeta_0, x_0) \in U_0 \times \mathbb{C}^{2n}$ and $(\zeta_1, x_1) \in U_1 \times \mathbb{C}^{2n}$ are equivalent if 
\[
\begin{cases}
\zeta_0 \cdot \zeta_1 = 1, \\
x_0^\alpha = \zeta_0 \cdot x_1^\alpha \quad \text{for } \alpha = 1, \ldots, 2n.
\end{cases}
\]

Similarly, we define $\overline{\bigoplus^{2n}\mathcal{O}(1)}=(V_0\times \mathbb{C}^{2n})\bigcup(V_1\times \mathbb{C}^{2n})/\sim$ as a holomorphic vector bundle over $\overline{\mathbb{CP}^1}$. It is equipped with local coordinates $(\eta_i,y_i^1\cdots y_i^{2n})$ on $V_i\times \mathbb{C}^{2n}, i=0,1$. In these coordinates, $(\eta_0, y_0) \in U_0 \times \mathbb{C}^{2n}$ and $(\eta_1, y_1) \in U_1 \times \mathbb{C}^{2n}$ are equivalent if 
\[
\begin{cases}
\eta_0 \cdot \eta_1 = 1, \\
y_0^\alpha = \eta_0 \cdot y_1^\alpha \quad \text{for } \alpha = 1, \ldots, 2n.
\end{cases}
\]

We define an antiholomorphic map $\mathscr{I}: \bigoplus^{2n}\mathcal{O}(1) \rightarrow \overline{\bigoplus^{2n}\mathcal{O}(1)}$ such that:
\[
\begin{array}{ccc}
\bigoplus^{2n}\mathcal{O}(1) & \stackrel{\mathscr{I}}{\longrightarrow} & \overline{\bigoplus^{2n}\mathcal{O}(1)} \\
\downarrow & \circlearrowleft & \downarrow \\
\mathbb{CP}^1 & \stackrel{\mathfrak{I}}{\longrightarrow} & \overline{\mathbb{CP}^1}
\end{array}
\]
Here, we could express $\mathscr{I}$ as $\mathscr{I}_0$ and $\mathscr{I}_1$:
\[
\begin{cases}
\mathscr{I}_0: U_0 \times \mathbb{C}^{2n} \rightarrow V_0 \times \mathbb{C}^{2n}, \quad (\zeta_0, x_0) \mapsto (\bar\zeta_0, \bar x_0) \\
\mathscr{I}_1: U_1 \times \mathbb{C}^{2n} \rightarrow V_1 \times \mathbb{C}^{2n}, \quad (\zeta_1, x_1) \mapsto (\bar\zeta_1, \bar x_1)
\end{cases}
\]
It is essential to note that the map $\mathscr{I} : \bigoplus^{2n}\mathcal{O}(1) \rightarrow \overline{\bigoplus^{2n}\mathcal{O}(1)}$ acts as an identity map in the context of smooth manifolds and possesses antiholomorphic properties. Similarly, the inverse map $\mathscr{I}^{-1}: \overline{\bigoplus^{2n}\mathcal{O}(1)} \rightarrow \bigoplus^{2n}\mathcal{O}(1)$ also exhibits antiholomorphic characteristics.

\begin{definition}
A holomorphic mapping \(f: \bigoplus^{2n}\mathcal{O}(1) \rightarrow \overline{\bigoplus^{2n}\mathcal{O}(1)}\) is considered compatible with \(\sigma\) if the following diagram commutes, and the restriction of \(f\) to each fiber is a linear isomorphism:
\[
\begin{array}{ccc}
\bigoplus^{2n}\mathcal{O}(1) & \stackrel{f}{\longrightarrow} & \overline{\bigoplus^{2n}\mathcal{O}(1)} \\
\downarrow & \circlearrowleft & \downarrow \\
\mathbb{CP}^1 & \stackrel{\sigma}{\longrightarrow} & \overline{\mathbb{CP}^1}
\end{array}
\]
We define \(\bar f := \mathscr{I}^{-1} \circ f \circ \mathscr{I}^{-1}: \overline{\bigoplus^{2n}\mathcal{O}(1)} \rightarrow \bigoplus^{2n}\mathcal{O}(1)\).
\end{definition}

After the definition of a holomorphic mapping compatible with $\sigma$ has been established, we now turn our attention towards its specific representation. This definition provides us with a framework for analyzing such mappings, especially in terms of their linear isomorphism properties on the fibers. Through this definition, we are equipped to delve deeper into the intrinsic structure of the mapping $f$. In the upcoming lemma, we will specifically explore a form of this mapping, showcasing how it can be represented via a constant matrix $A$. 

\begin{lemma}
\label{holomorphic map compatible with sigma}
Let \(f: \bigoplus^{2n}\mathcal{O}(1) \rightarrow \bigoplus^{2n}\overline{\mathcal{O}(1)}\) be a holomorphic mapping that is compatible with \(\sigma\). Then \(f\) can be represented as \(f_0\) and \(f_1\) defined as follows:
\[
\begin{cases}
f_0: U_0 \times \mathbb{C}^{2n} \rightarrow V_1 \times \mathbb{C}^{2n} & (\zeta_0, x_0) \mapsto (-\zeta_0, -A \cdot x_0) \\
f_1: U_1 \times \mathbb{C}^{2n} \rightarrow V_0 \times \mathbb{C}^{2n} & (\zeta_1, x_1) \mapsto (-\zeta_1, A \cdot x_1)
\end{cases}
\]
where \(A \in GL(2n, \mathbb{C})\) is a constant matrix. Furthermore, if \(f \circ \bar f = \text{id}\), then the matrix \(A\) satisfies \(A \cdot \bar{A} = -\mathcal{I}_{2n}\).
\end{lemma}

\begin{proof}
Consider a holomorphic map $f: \bigoplus^{2n}\mathcal{O}(1) \rightarrow \overline{\bigoplus^{2n}\mathcal{O}(1)}$, which is compatible with $\sigma$. In this context, $f$ can be decomposed into $(f_0, f_1)$, where:
\[
\begin{cases}
f_0: U_0 \times \mathbb{C}^{2n} \rightarrow V_1 \times \mathbb{C}^{2n}, & (\zeta_0, x_0) \mapsto \left(-\zeta_0, A_0(\zeta_0) \cdot x_0\right),\\
f_1: U_1 \times \mathbb{C}^{2n} \rightarrow V_0 \times \mathbb{C}^{2n}, & (\zeta_1, x_1) \mapsto \left(-\zeta_1, A_1(\zeta_1) \cdot x_1\right).
\end{cases}
\]
Here, $A_0: U_0 \cong \mathbb{C} \rightarrow GL(2n, \mathbb{C})$ and $A_1: U_1 \cong \mathbb{C} \rightarrow GL(2n, \mathbb{C})$ are holomorphic maps. We impose an additional condition that
\[
(\zeta_0, x_0) \sim (\zeta_1, x_1) \in \bigoplus^{2n}\mathcal{O}(1) \iff f(\zeta_0, x_0) \sim f(\zeta_1, x_1) \in \overline{\bigoplus^{2n}\mathcal{O}(1)},
\]
which implies that $-A_1(\zeta_1) = A_0(\zeta_1^{-1})$ for any $\zeta_1 \in U_0 \cap U_1 \cong \mathbb{C}^*$. Consequently, $A_0: U_0 \cong \mathbb{C} \rightarrow \mathbb{C}^*$ is a holomorphic matrix-valued function with bounded limits at infinity, indicating that $A_0$ is constant and $A_1 = -A_0$. Let us denote $A_1 = -A_0 = A$.

Define $f$ as the above map. Then, construct a holomorphic map $\bar{f} = \mathscr{I}^{-1} \circ f \circ \mathscr{I}^{-1}: \overline{\bigoplus^{2n}\mathcal{O}(1)} \rightarrow \bigoplus^{2n}\mathcal{O}(1)$. $\bar{f}$ can be represented by $\bar{f}_0$ and $\bar{f}_1$ as:
\[
\begin{cases}
\bar{f}_0: V_0 \times \mathbb{C}^{2n} \rightarrow U_1 \times \mathbb{C}, & (\eta_0, y_0) \mapsto \left(-\eta_0, -\bar{A} \cdot y_0\right),\\
\bar{f}_1: V_1 \times \mathbb{C}^{2n} \rightarrow U_0 \times \mathbb{C}, & (\eta_1, y_1) \mapsto \left(-\eta_1, \bar{A} \cdot y_1\right).
\end{cases}
\]
These representations utilize the same trivializations as those defined for the map $f$. A direct calculation shows that $f \circ \bar{f} = \text{id}$ if and only if $A \cdot \bar{A} = -\mathcal{I}_{2n}$.
\end{proof}

Consider a holomorphic map $f\colon \bigoplus^{2n}\mathcal{O}(1) \rightarrow \overline{\bigoplus^{2n}\mathcal{O}(1)}$ that is compatible with $\sigma$. Let $s$ be a holomorphic section of $\bigoplus^{2n}\mathcal{O}(1)$. We aim to examine certain holomorphic sections of $\overline{\bigoplus^{2n}\mathcal{O}(1)}$ that are induced by $s$. Commonly, there are two types of such induced sections: $\bar{s}$, defined as $\bar{s} := \mathscr{I} \circ s \circ \mathfrak{I}^{-1}$, and $\tilde{s}$, defined as $\tilde{s} := f \circ s \circ \sigma^{-1}$. We will analyze these sections in detail as follows.

Suppose $s\colon \mathbb{CP}^1 \rightarrow \bigoplus^{2n}\mathcal{O}(1)$ is a holomorphic section of the vector bundle $\bigoplus^{2n}\mathcal{O}(1)$. This section can be represented as $s_0$ and $s_1$ with respect to the trivialization $\bigoplus^{2n}\mathcal{O}(1) = (U_0 \times \mathbb{C}^{2n}) \cup (U_1 \times \mathbb{C}^{2n})/\sim$, as follows:
\[
\begin{cases}
s_0: U_0 \rightarrow U_0 \times \mathbb{C}^{2n}, & \zeta_0 \mapsto (\zeta_0, a + b \cdot \zeta_0),\\
s_1: U_1 \rightarrow U_1 \times \mathbb{C}^{2n}, & \zeta_1 \mapsto (\zeta_1, a \cdot \zeta_1 + b),
\end{cases}
\]
where $a, b \in \mathbb{C}^{2n}$. Considering the trivialization of $\bigoplus^{2n}\mathcal{O}(1)$ used to represent $s$, we define a $\mathbb{C}$-linear mapping
\[
\varphi: \mathbb{C}^{4n} \rightarrow H^0\left(\mathbb{CP}^1, \bigoplus^{2n}\mathcal{O}(1)\right); (a, b) \mapsto s,
\]
which is an isomorphism given that $\text{dim}_{\mathbb{C}} H^0(\mathbb{CP}^1, \bigoplus^{2n}\mathcal{O}(1)) = 4n$. Similarly, a $\mathbb{C}$-linear isomorphism $\psi: \mathbb{C}^{4n} \rightarrow H^0\left(\overline{\mathbb{CP}^1}, \overline{\bigoplus^{2n}\mathcal{O}(1)}\right)$ can be defined. The isomorphisms $\varphi$ and $\psi$ depending on the choice of trivialization.

Let $s$ be as defined above and $f$ be the map given in Lemma \ref{holomorphic map compatible with sigma}. We define $\bar{s} := \mathscr{I} \circ s \circ \mathfrak{I}^{-1}\colon \overline{\mathbb{CP}^1} \rightarrow \overline{\bigoplus^{2n}\mathcal{O}(1)}$ and the section $\bar{s}$ can be expressed as $\bar{s}_0$ and $\bar{s}_1$ as follows:
\[
\begin{cases}
\bar{s}_0: V_0 \rightarrow V_0 \times \mathbb{C}^{2n}, & \eta_0 \mapsto \left(\eta_0, \bar{a} + \bar{b} \cdot \eta_0\right),\\
\bar{s}_1: V_1 \rightarrow V_1 \times \mathbb{C}^{2n}, & \eta_1 \mapsto \left(\eta_1, \bar{a} \cdot \eta_1 + \bar{b}\right),
\end{cases}
\]
where $a, b \in \mathbb{C}^{2n}$. Define $\tilde{s} := f \circ s \circ \sigma^{-1}\colon \overline{\mathbb{CP}^1} \rightarrow \overline{\bigoplus^{2n}\mathcal{O}(1)}$ and the section $\tilde{s}$ can be expressed as $\tilde{s}_0$ and $\tilde{s}_1$ as follows:
\[
\begin{cases}
\tilde{s}_0: V_0 \rightarrow V_0 \times \mathbb{C}^{2n}, & \eta_0 \mapsto \left(\eta_0, A \cdot \left(-a \cdot \eta_0 + b\right)\right),\\
\tilde{s}_1: V_1 \rightarrow V_1 \times \mathbb{C}^{2n}, & \eta_1 \mapsto \left(\eta_1, A \cdot \left(-a + b \cdot \eta_1\right)\right),
\end{cases}
\]
where $a, b \in \mathbb{C}^{2n}$ and $A$ is the matrix in the lemma. This demonstrates that both $\bar{s}$ and $\tilde{s}$ are holomorphic sections of $\overline{\bigoplus^{2n}\mathcal{O}(1)}$.

If $s = \varphi(a, b)$ represents a holomorphic section of $\bigoplus^{2n}\mathcal{O}(1)$, then $\tilde{s}$ and $\bar{s}$ represent holomorphic sections of $\overline{\bigoplus^{2n}\mathcal{O}(1)}$ and can be expressed as $\psi(A \cdot b, -A \cdot a)$ and $\psi(\bar{a}, \bar{b})$, respectively. Conversely, if $\mathfrak{s} = \psi(a, b)$ is a holomorphic section of $\overline{\bigoplus^{2n}\mathcal{O}(1)}$, then $\mathfrak{\tilde{s}}$ and $\mathfrak{\bar{s}}$ are holomorphic sections of $\bigoplus^{2n}\mathcal{O}(1)$, given by $\mathfrak{\tilde{s}} := \bar{f} \circ s \circ \bar{\sigma}$ and $\mathfrak{\bar{s}} := \mathscr{I}^{-1} \circ s \circ \mathfrak{I}$, which can be expressed as $\varphi(\bar{A} \cdot b, -\bar{A} \cdot a)$ and $\varphi(\bar{a}, \bar{b})$, respectively.

In summary, we present the following facts:

\begin{lemma}
Let $s$ be a holomorphic section of $\bigoplus^{2n}\mathcal{O}(1)$, and let $f$ be a holomorphic map compatible with $\sigma$. Then, $f$ induces a linear automorphism $r_f$ on $H^0\left(\mathbb{CP}^1, \bigoplus^{2n}\mathcal{O}(1)\right)$, defined by 
\[ s \mapsto \mathscr{I}^{-1} \circ f \circ s \circ \sigma^{-1} \circ \mathfrak{I}. \]
Given a trivialization of $\bigoplus^{2n}\mathcal{O}(1) = (U_0 \times \mathbb{C}^{2n}) \cup (U_1 \times \mathbb{C}^{2n})/\sim$, corresponding to a standard open cover of $\mathbb{CP}^1$, we can associate $f$ with a constant matrix $A \in GL(2n, \mathbb{C})$ as detailed in Lemma \ref{holomorphic map compatible with sigma}. Consequently, if $s = \varphi(a, b)$ with respect to this trivialization, then $r_f(s) = \varphi\left(\bar A \cdot \bar b, -\bar A \cdot \bar a\right)$.
\end{lemma}

\begin{remark}
Let $s$ be a holomorphic section of $\bigoplus^{2n}\mathcal{O}(1)$. We define $\mathfrak{t}=\tilde{s}=f \circ s \circ \sigma^{-1}$ as a holomorphic section of $\overline{\bigoplus^{2n}\mathcal{O}(1)}$. According to the provided definition, it can be observed that $r_f(s) = \mathfrak{\bar{t}}$. Simultaneously, if we define $\mathfrak{r}=\bar{s} = \mathscr{I}\circ s\circ \mathfrak{I}^{-1}$ as a holomorphic section of $\overline{\bigoplus^{2n}\mathcal{O}(1)}$, then in this context, $r_f(s)$ is equivalent to $\mathfrak{\tilde{r}}$.
\end{remark}

\begin{definition}
\label{real structure and quaternionic structure}
Let $V$ be a complex vector space. A conjugate linear map on $V$, denoted as $\gamma: V\rightarrow V$, is defined as follows: 
\[
\gamma(\lambda\cdot v+\mu\cdot w)=\bar{\lambda}\cdot \gamma(v)+\bar{\mu}\cdot \gamma(w)
\]
for all $\lambda, \mu\in\mathbb{C}$ and $v,w\in V$.
\begin{itemize}
  \item A \textit{real structure} on $V$ is a conjugate linear map $r: V\rightarrow V$ that satisfies $r^2=\text{id}_V$.
  \item A \textit{quaternionic structure} on $V$ is a conjugate linear map $j: V\rightarrow V$ that satisfies $j^2=-\text{id}_V$.
\end{itemize}
\end{definition}

Let $V$ be a complex vector space with a quaternionic structure denoted by $j: V \rightarrow V$. Then, $V$ is equipped with a hypercomplex structure characterized by three linear maps: $\mathbf{I}(v) = \sqrt{-1} \cdot v$, $\mathbf{J}(v) = j(v)$, and $\mathbf{K}(v) = \sqrt{-1} \cdot j(v)$. These maps satisfy the quaternionic relations $\mathbf{I}^2 = \mathbf{J}^2 = \mathbf{K}^2 = -\textrm{id}_V$ and the compatibility condition $\mathbf{K} = \mathbf{I} \circ \mathbf{J} = -\mathbf{J} \circ \mathbf{I}$. It is important to note that for such a quaternionic structure to exist, $V$ must have an even complex dimension.

\begin{definition}
\label{Real structure}
A \textit{real structure} on $\bigoplus^{2n}\mathcal{O}(1)$ is a holomorphic map $f: \bigoplus^{2n}\mathcal{O}(1)\rightarrow\overline{\bigoplus^{2n}\mathcal{O}(1)}$ that is compatible with $\sigma$ and satisfies $f\circ\bar{f}=\text{id}$.
\end{definition}

We will now establish a correspondence between real structures on the vector bundle $\bigoplus^{2n}\mathcal{O}(1)$ and quaternionic structures on 
\[
H^0\left(\mathbb{CP}^1,\left(\bigoplus^{2n}\mathcal{O}(1)\right)\otimes\mathcal{O}(-1)\right) = H^0\Big(\mathbb{CP}^1,\mathbb{CP}^1\times\mathbb{C}^{2n}\Big) \cong \mathbb{C}^{2n}.
\]
which will play a crucial role in establishing hyperkähler structures in the following section.

\begin{prop}
\label{main section2}
\[
\begin{cases}
\mathscr{R} := \Bigl\{\text{real structures on }\bigoplus^{2n}\mathcal{O}(1)\Bigr\} \\
\mathcal{R} := \left\{\text{real structures on }H^0\Bigl(\mathbb{CP}^1, \bigoplus^{2n}\mathcal{O}(1)\Bigr)\right\} \\
\mathcal{J} := \left\{\text{quaternionic structures on }H^0\Bigl(\mathbb{CP}^1,\Bigl(\bigoplus^{2n}\mathcal{O}(1)\Bigr)\otimes\mathcal{O}(-1)\Bigr)\right\}
\end{cases}
\]
There exists an injection from $\mathscr{R}$ to $\mathcal{R}$. Moreover, there exists a one-to-one correspondence between $\mathscr{R}$ and $\mathcal{J}$.
\end{prop}

\begin{proof} 
We consider trivializations of $\bigoplus^{2n}\mathcal{O}(1) = (U_0 \times \mathbb{C}^{2n}) \bigcup (U_1 \times \mathbb{C}^{2n})/\sim$ and $\mathcal{O}(-1) = (U_0 \times \mathbb{C}) \bigcup (U_1 \times \mathbb{C})/\sim$ with respect to the standard open cover of $\mathbb{CP}^1$. 

In this trivialization, any holomorphic section of $\bigoplus^{2n}\mathcal{O}(1)$ can be represented as $\varphi(a,b)$, and any holomorphic section of $\left(\bigoplus^{2n}\mathcal{O}(1)\right)\otimes\mathcal{O}(-1)$ can be represented by the constant $x\in\mathbb{C}^{2n}$. Suppose $f$ is a real structure on $\bigoplus^{2n}\mathcal{O}(1)$, defined by $f_0$ and $f_1$ under the given trivialization. We can find a constant matrix $A \in GL(2n, \mathbb{C})$ such that $A \bar{A} = -\mathcal{I}_{2n}$, and
\[
\begin{cases}
f_0: U_0 \times \mathbb{C}^{2n} \rightarrow V_1 \times \mathbb{C}^{2n} & (\zeta_0, x_0) \mapsto (-\zeta_0, -A x_0), \\
f_1: U_1 \times \mathbb{C}^{2n} \rightarrow V_0 \times \mathbb{C}^{2n} & (\zeta_1, x_1) \mapsto (-\zeta_1, A x_1).
\end{cases}
\]
Then $f$ corresponds to the real structure $r_f$ on $H^0\Bigl(\mathbb{CP}^1, \bigoplus^{2n}\mathcal{O}(1)\Bigr)$ and the quaternionic structure $j_f$ on $H^0\Bigl(\mathbb{CP}^1,\Bigl(\bigoplus^{2n}\mathcal{O}(1)\Bigr)\otimes\mathcal{O}(-1)\Bigr)$ as follows: For a holomorphic section $s = \varphi(a, b)$ of $\bigoplus^{2n}\mathcal{O}(1)$ and a holomorphic section $x\in\mathbb{C}^{2n}$ of the bundle $\left(\bigoplus^{2n}\mathcal{O}(1)\right)\otimes \mathcal{O}(-1)$ with respect to the given trivialization, define
\[
\begin{cases}
r_f(s) := \overline{f \circ s \circ \sigma^{-1}} = \varphi\left(\bar{A} \bar{b}, -\bar{A} \bar{a}\right), \\
j_f(x) := \bar A\bar x
\end{cases}
\]
On the other hand, for any quaternionic structure $j$ on $\mathbb{C}^{2n}$, there exists a matrix $A\in GL(2n,\mathbb{C})$ such that $j(x)$ is given by $\bar A\cdot \bar x$. We can then construct $f_0$ and $f_1$ similarly as mentioned above to represent a real structure $f$ on $\bigoplus^{2n}\mathcal{O}(1)$. The mapping from $f$ to $r_f$ is invariant under different choices of trivialization.

Next, we will demonstrate that the correspondences between $f$ and $j_f$ are also unaffected by the choice of trivialization of $\bigoplus^{2n}\mathcal{O}(1)$ in the following manner:

If another trivialization exists for $\bigoplus^{2n}\mathcal{O}(1)=(U_0\times(\mathbb{C}^{2n})^\prime)\bigcup(U_1\times(\mathbb{C}^{2n})^\prime)/\sim$ with respect to the standard open cover of $\mathbb{CP}^1$, let $(\zeta_0,x_0^\prime)$ and $(\zeta_1,x_1^\prime)$ be the coordinates on $(U_0\times(\mathbb{C}^{2n})^\prime)$ and $(U_1\times(\mathbb{C}^{2n})^\prime)$. Then there exist matrix-valued holomorphic functions $P_i:U_i\rightarrow GL(r,2n)$ such that $(\zeta_i,x_i)$ and $(\zeta_i,x_i^\prime)$ represent the same point in $\mathbb{CP}^1$ if and only if $x_i=P_i(\zeta_i)\cdot x_i^\prime$, $i=0,1$. Additionally, $(\zeta_0,x_0)\sim(\zeta_1,x_1)$ if and only if $(\zeta_0,x_0^\prime)\sim(\zeta_1,x_1^\prime)$, implying that 
\[
\begin{cases}
\zeta_0\cdot\zeta_1=1\\
x_0=\text{Diag}(\zeta_0,\cdots,\zeta_0)\cdot x_1
\end{cases}\Leftrightarrow\quad
\begin{cases}
\zeta_0\cdot\zeta_1=1\\
x_0^\prime=\text{Diag}(\zeta_0,\cdots,\zeta_0)\cdot x_1^\prime
\end{cases}
\]
This implies that $P_0(\zeta_0)^{-1}\cdot P_1(\zeta_0^{-1})=\mathcal{I}_{2n}$ for $\zeta_0\in\mathbb{C}^*$. Hence, $P_0(\zeta_0)$ is a holomorphic function on $U_0\cong\mathbb{C}$ with finite limits at infinity, indicating that $P_0$ is a constant function. We denote $P_1=P_0:=P$. 

Under this new local trivialization, the real structure $f$ also induces a matrix $A^\prime$. It can be observed that $A^\prime$ is equal to $\bar P\cdot A\cdot P^{-1}$. Additionally, in the previous trivialization, $x\in\mathbb{C}^n$ is transformed to $P\cdot x$ in the new trivialization. This is further mapped to $\bar A^\prime\cdot \overline{P\cdot x}=P\cdot\bar A\cdot\bar x$. Notably, this precisely represents the coordinates of $j_f(x)$ in the new trivialization.
\end{proof}

\begin{remark}
\label{trivialization of O(-1)}
It is worth noting that the definition of $j_f(s)$ is not independent of the chosen trivialization of $\mathcal{O}(-1)$. Therefore, when discussing $j_f(s)$, it is essential to specify the trivialization of $\mathcal{O}(-1)$ that was used.
\end{remark}

Let $f$ be a real structure on $\bigoplus^{2n}\mathcal{O}(1)$. We choose a trivialization $(U_0 \times \mathbb{C}^{2n}) \bigcup (U_1 \times \mathbb{C}^{2n})/\sim$ for $\bigoplus^{2n}\mathcal{O}(1)$ and a trivialization for $\mathcal{O}(-1)$ based on the standard open cover of $\mathbb{CP}^1$. The maps $r_f$ and $j_f$ induced by $f$ have the following relationship:

\begin{cor}
\label{induced real and quaternionic structures}
For any holomorphic section $s=\varphi(a,b)$ in $H^0\Bigl(\mathbb{CP}^1, \bigoplus^{2n}\mathcal{O}(1)\Bigr)$ under this chosen trivialization, we can express $r_f(s)$ as $\varphi\left(j_f(b),-j_f(a)\right)$. Moreover, we can express $f$ as follows: 
\[
\begin{cases}
f_0: U_0 \times \mathbb{C}^{2n} \rightarrow V_1 \times \mathbb{C}^{2n} & (\zeta_0, x_0) \mapsto \left(-\zeta_0, -\overline{j_f(x_0)}\right), \\
f_1: U_1 \times \mathbb{C}^{2n} \rightarrow V_0 \times \mathbb{C}^{2n} & (\zeta_1, x_1) \mapsto \left(-\zeta_1, \overline{j_f(x_1)}\right).
\end{cases}
\]
\end{cor}


\section{Twistor spaces}

In this section, we examine the sophisticated interconnection between holomorphic fiber bundles and hyperkähler manifolds, guided by the framework established by Hitchin, Karlhede, Lindstr{\"o}m, and Ro{\v{c}}ek \cite[Section 3.F]{hitchin1987hyperkahler}. We describe $\mathcal{Z}$ as a twistor space, stipulating it as a complex manifold with dimension $2n+1$ that satisfies a set of predefined conditions:

\begin{enumerate}
\item $\mathcal{Z}$ forms a holomorphic fiber bundle over $\mathbb{CP}^1$, represented by $p:\mathcal{Z}\rightarrow \mathbb{CP}^1$.
\item A real structure $\tau$ exists on $\mathcal{Z}$, compatible with the antipodal map $\sigma$.
\item $\mathcal{Z}$ supports a family of holomorphic sections $\mathcal{M}$, each with a normal bundle $T_Fs\cong\bigoplus^{2n}\mathcal{O}(1)$.
\item A specific holomorphic section $\omega$ of $\Lambda^2{(N\mathcal{Z})}^\otimes p^*\mathcal{O}(2)$ is present.
\end{enumerate}

In accordance with conditions (1), we define a complex manifold $\mathcal{M}$ that acts as the parameter space for these holomorphic sections, each coupled with a normal bundle of $\bigoplus^{2n}\mathcal{O}(1)$, as elaborated in Proposition \ref{open submanifold}. Building upon conditions (1), (2), and (3), the existence of a smooth manifold $M$, endowed with a hypercomplex structure and integrated as a smooth submanifold of $\mathcal{M}$, is established, as validated by Thoerem \ref{hypercomplex}. Finally, under the comprehensive conditions (1), (2), (3), and (4), Theorem \ref{main section4} presents a methodology for constructing a hyperkähler manifold $M$ of real dimension $4n$. This manifold serves as the parameter space for real sections of the fiber bundle $p:\mathcal{Z}\rightarrow\mathbb{CP}^1$, elucidating the deep connection between complex geometry and hyperkähler structures.

In the fiber bundle $\mathcal{Z}$, the vertical space is denoted as $N\mathcal{Z}$ and defined as the kernel of the projection map $\textrm{d}p: T\mathcal{Z} \rightarrow T\mathbb{CP}^1$. The antipodal map on $\mathbb{CP}^1$, denoted by $\sigma$, is defined in section 3. Specifically, $\bar{\mathcal{Z}}$ is a complex manifold that is topologically equivalent to $\mathcal{Z}$ but possesses another complex structure denoted as $\mathbf{I}_{\bar{\mathcal{Z}}} := -\mathbf{I}_{\mathcal{Z}}$. The map $\mathfrak{I}: \mathcal{Z} \rightarrow \bar{\mathcal{Z}}$ acts as an identity on the smooth manifold on $\mathcal{Z}$ and as an antiholomorphic map from the complex manifold $\mathcal{Z}$ to $\bar{\mathcal{Z}}$. Furthermore, the conjugation from $T^{1,0}\mathcal{Z}$ to $T^{0,1}\mathcal{Z}$ transforms the vertical space $N\mathcal{Z}$ into $\overline{N\mathcal{Z}}$, which then serves as the vertical space for the projection $\bar{p}: \bar{\mathcal{Z}} \rightarrow \overline{\mathbb{CP}^1}$.

\begin{definition}
A \textit{real structure} on $\mathcal{Z}$ is a holomorphic map $\tau: \mathcal{Z} \rightarrow \bar{\mathcal{Z}}$ satisfying $\tau \circ \bar{\tau} = \text{id}_\mathcal{Z}$, where $\bar{\tau} := \mathfrak{I}^{-1} \circ \tau \circ \mathfrak{I}:\bar{\mathcal{Z}}\rightarrow\mathcal{Z}$. In this context, $\bar{\tau}$ is an antiholomorphic map from $\bar{\mathcal{Z}}$ to $\mathcal{Z}$, and it aligns with $\tau$ as a smooth map on $\mathcal{Z}$.
\end{definition}

\begin{definition}
A real structure $\tau$ on $\mathcal{Z}$ is \textit{compatible with $\sigma$} if it satisfies the commutativity of the following diagram: 
\begin{align*}
\begin{array}{ccc}
\mathcal{Z} & \stackrel{\tau}{\longrightarrow} & \bar{\mathcal{Z}} \\
\downarrow & \circlearrowleft & \downarrow \\
\mathbb{CP}^1 & \stackrel{\sigma}{\longrightarrow} & \overline{\mathbb{CP}^1}
\end{array}
\end{align*}
\end{definition}

Let $\tau$ denote a real structure on $\mathcal{Z}$, and consider a holomorphic section $s: \mathbb{CP}^1 \rightarrow \mathcal{Z}$. Consequently, the image $s(\mathbb{CP}^1)$ forms a complex submanifold within $\mathcal{Z}$. Following this, the normal bundle $Ns(\mathbb{CP}^1)$ of $s(\mathbb{CP}^1)$ in $\mathcal{Z}$ is examined in accordance with Definition \ref{normal bundle}. Specifically, we observe that $\left.T\mathcal{Z}\right|_{s(\mathbb{CP}^1)}$ equals $\left.N\mathcal{Z}\right|_{s(\mathbb{CP}^1)} \oplus Ts(\mathbb{CP}^1)$, indicating that $Ns(\mathbb{CP}^1)$ is equivalent to $\left.N\mathcal{Z}\right|_{s(\mathbb{CP}^1)}$. For brevity, we adopt the notation $T_Fs$ to denote $\left.N\mathcal{Z}\right|_{s(\mathbb{CP}^1)}$, which we shall refer to as the \textit{normal bundle of section $s$}. If $\tau$ represents a real structure on $\mathcal{Z}$ that is compatible with $\sigma$, as indicated by the relation $\bar{p} \circ \tau = \sigma \circ p$, then it follows that $\textrm{d}\bar{p} \circ \textrm{d}\tau = \textrm{d}\sigma \circ \textrm{d}p$. From this, we can infer that $\textrm{d}\tau(N\mathcal{Z}) = \overline{N\mathcal{Z}}$.

\begin{definition}
\label{real sections}
Let $\tau$ be a real structure on $\mathcal{Z}$, compatible with $\sigma$. A holomorphic section $s$ of $\mathcal{Z}$ is termed a \textit{real section} if it satisfies the commutativity of the following diagram:
\begin{align*}
\begin{array}{ccc}
\mathcal{Z} & \stackrel{\tau}{\longrightarrow} & \bar{\mathcal{Z}} \\
s\hspace{-0.3em}\uparrow{~~~~~~}
 & \circlearrowleft & {~~~~}\uparrow\hspace{-0.3em}\bar s \\
\mathbb{CP}^1 & \stackrel{\sigma}{\longrightarrow} & \overline{\mathbb{CP}^1}
\end{array}
\end{align*}
where $\bar s$ is identified with $s$ as smooth sections of the smooth fiber bundle $\mathcal{Z}\to \mathbb{CP}^1$ and it is also regarded as a holomorphic section of the holomorphic fiber bundle $\mathcal{\bar Z}\to \overline{\mathbb{CP}^1}$.
\end{definition}

Consider a holomorphic section $s$ of $\mathcal{Z}$ with the normal bundle $T_Fs \cong \bigoplus^{2n}\mathcal{O}(1)$. Define $\tilde s := \tau \circ s \circ \sigma^{-1}$, which is a holomorphic section of $\bar{\mathcal{Z}}$. In this context, the vector bundle $\textrm{d}\tau(T_Fs)$ precisely aligns with the normal bundle $T_F\tilde s$ over $\tilde s(\overline{\mathbb{CP}^1})$ in $\bar{\mathcal{Z}}$, given that $\textrm{d}\tau(N\mathcal{Z})$ equals $\overline{N\mathcal{Z}}$. Furthermore, since $\tau$ is a biholomorphic map between $\mathcal{Z}$ and $\bar{\mathcal{Z}}$, it follows that $T_F\tilde s = \textrm{d}\tau(T_Fs) \cong \overline{\bigoplus^{2n}\mathcal{O}(1)}$. Importantly, when $s$ is a real section, this indicates that $\tilde s$ corresponds to both the smooth section $s$ of $\mathcal{Z}$ and the holomorphic section $\bar s$ of $\bar{\mathcal{Z}}$. As a result, $\textrm{d}\tau(T_Fs) = T_F\tilde s$ coincides with $\overline{T_Fs}$.

\begin{lemma}
Let $\tau$ be a real structure on $\mathcal{Z}$ compatible with $\sigma$, and consider $s$ as a real section of $\mathcal{Z}$ with the normal bundle $T_Fs \cong \bigoplus^{2n}\mathcal{O}(1)$. Then, the mapping $\textrm{d}\tau: T_Fs \rightarrow \overline{T_Fs}$ constitutes a real structure on $T_Fs$, as delineated in Definition \ref{Real structure}.
\end{lemma}

\begin{proof}
For each point $\zeta \in \mathbb{CP}^1$, $\textrm{d}\tau$ maps a vector in $T_Fs$ at $s(\zeta)$ to a vector in $\overline{T_Fs}$ at $\tau(s(\zeta)) = \bar s(\sigma(\zeta))$. This demonstrates that $\textrm{d}\tau$ is compatible with $\sigma$. Furthermore, the condition $\tau \circ \bar\tau = \textrm{id}_\mathcal{Z}$ implies that $\textrm{d}\tau \circ \overline{\textrm{d}\tau} = \textrm{id}$.
\end{proof}

Based on this, we can now provide the following definition:

\begin{definition}
\label{induced real structure and quaternionic structure}
Given a real structure $\tau$ on $\mathcal{Z}$ compatible with $\sigma$, and a real section $s$ of $\mathcal{Z}$ with the normal bundle $T_Fs \cong \bigoplus^{2n}\mathcal{O}(1)$, we define $r(\tau, s)$ as the real structure on $H^0\big(s(\mathbb{CP}^1), T_Fs\big)$. Similarly, $j(\tau, s)$ is defined as the quaternionic structure on $H^0\big(s(\mathbb{CP}^1), T_Fs\otimes\mathcal{O}(-1)\big)$, both induced by the real structure $\textrm{d}\tau$ on $T_Fs$ by proposition \ref{main section2}.
\end{definition}
  
As highlighted in Remark \ref{trivialization of O(-1)}, when discussing $j(\tau, s)$, it is crucial to specify a fixed trivialization of $\mathcal{O}(-1)$ for each real holomorphic section $s$. This is accomplished as follows: We initially establish a trivialization of the vector bundle $\varpi: \mathcal{O}(-1) \rightarrow \mathbb{CP}^1$. Then, considering the map $p: s(\mathbb{CP}^1) \rightarrow \mathbb{CP}^1$ as a biholomorphism and the pullback bundle $p^*(\mathcal{O}(-1))$ on $s(\mathbb{CP}^1)$ is endowed with a natural trivialization. This trivialization is induced by the original bundle $\varpi: \mathcal{O}(-1) \rightarrow \mathbb{CP}^1$ for each real holomorphic section $s$. In subsequent discussions regarding the bundle $\mathcal{O}(-1)$ over $s(\mathbb{CP}^1)$, it is assumed that this particular trivialization is applied.

We define $s$ as a holomorphic section of $\mathcal{Z}$, where $T_Fs \cong \bigoplus^{2n}\mathcal{O}(1)$. The image $s(\mathbb{CP}^1)$ forms a compact complex submanifold of $\mathcal{Z}$. Furthermore, we consider the regular deformation space $\widetilde{\mathscr{M}^{4n}_{s(\mathbb{CP}^1)}}$ within $\mathcal{Z}$, as described in Definition \ref{regular deformation space}. It has been established that $\widetilde{\mathscr{M}^{4n}_{s(\mathbb{CP}^1)}}$ is a complex manifold of dimension $4n$ in theorem \ref{complex regular deformation space}. We now proceed to define
\[
\mathcal{M}:=\left\{V \in \widetilde{\mathscr{M}^{4n}_{s(\mathbb{CP}^1)}} | V \text{ is a holomorphic section of } \mathcal{Z}; T_FV \cong \bigoplus^{2n}\mathcal{O}(1)\right\}
\]

\begin{lemma}
\label{open submanifold}
$\mathcal{M}$ constitutes an open submanifold within $\widetilde{\mathscr{M}^{4n}_{s(\mathbb{CP}^1)}}$.
\end{lemma}

\begin{proof}

Let $s$ belong to $\mathcal{M}$. Consider the standard open cover $U_0 \cup U_1$ of $\mathbb{CP}^1$, with each $U_i$ isomorphic to $\mathbb{C}$ and possessing a local coordinate $\zeta_i$ for $i = 0,1$. An open cover $\mathscr{U}_0 \cup \mathscr{U}_1$ of $s(\mathbb{CP}^1)$ in $\mathcal{Z}$ exists, where each $\mathscr{U}_i$ is represented as $\mathbb{B}_i\times U_i$, with $\mathbb{B}_i$ being open subsets of $\left(\mathbb{C}^{2n};z_i^1,\cdots,z_i^{2n}\right)$ for $i = 0,1$. Furthermore, $\mathscr{U}_i$ serves as local charts on $\mathcal{Z}$ with coordinates $(z_i^1,\cdots, z_i^{2n},\zeta_i)$. Define $z_i=(z_i^1,\cdots,z_i^{2n})$. The bundle map $p$ is explicitly represented in these coordinates as: $p:\mathscr{U}_i\rightarrow U_i;(z_i,\zeta_i)\mapsto \zeta_i$, for $i=0,1$.

Without loss of generality, we assume that $s(\mathbb{CP}^1) \cap \mathscr{U}_i = \{z_i = 0\}$ for each $i = 0,1$. The transition functions between $\mathscr{U}_0$ and $\mathscr{U}_1$ are given by:
\[
\begin{cases}
\zeta_0 \cdot \zeta_1 = 1 & \text{on } U_0 \cap U_1 \cong \mathbb{C}^* \\
z_1^\lambda = f^\lambda(z_0,\zeta_0) & \text{for } \lambda = 1,\cdots, 2n
\end{cases}
\]
Here, $\mathscr{V}_i$ denotes $\mathscr{U}_i \cap s(\mathbb{CP}^1)$ for $i = 0,1$. Considering that $H^1(s(\mathbb{CP}^1), T_Fs) = 0$, there exists an analytic family $\pi: \mathcal{V} \rightarrow \mathbb{B}^{4n}_\varepsilon$ of compact submanifolds which represents canonical local deformations of $s(\mathbb{CP}^1)$ within $\mathcal{M}$ and is characterized by the following properties:
\begin{enumerate}
    \item $V_0 = s(\mathbb{CP}^1)$.
    \item The Kodaira-Spencer map $\sigma_t$ of $\pi$ is a linear isomorphism for all $t \in \mathbb{B}^{4n}_\varepsilon$.
    \item $\pi$ is maximal at each point $t \in \mathbb{B}^{4n}_\varepsilon$.
\end{enumerate}
Moreover, $\pi^{-1}(\mathbb{B}_\varepsilon^{4n})$ is an open set in $\widetilde{\mathscr{M}^{4n}_{s(\mathbb{CP}^1)}}$. When $\varepsilon$ is sufficiently small, we only need to show $\pi^{-1}(\mathbb{B}_\varepsilon^{4n}) \subset \mathcal{M}$. Vector valued holomorphic functions $\varphi_i: \mathscr{V}_i \times \mathbb{B}^{4n}_\varepsilon \rightarrow \mathbb{C}^{2n}$, defined as $\varphi_i(\zeta_i, t^1, \cdots, t^{4n}) = (\varphi^1_i, \cdots, \varphi^{2n}_i)$ for $i = 0,1$, satisfy the following conditions:

Transition conditions:
\[
\begin{cases}
\zeta_0 \cdot \zeta_1 = 1 & \text{ on } U_0 \cap U_1 \cong \mathbb{C}^* \\
\varphi_1^\lambda = f^\lambda(\varphi^1_0, \cdots, \varphi^{2n}_0, \zeta_0) & \text{ for } \lambda = 1, \cdots, 2n
\end{cases}
\]

Boundary conditions:
\[
\begin{cases}
\varphi_i^\lambda(\zeta_i, 0) = 0 & i = 0, 1 \\
\left.\frac{\partial \varphi^\lambda_i}{\partial t^\rho}\right|_{t = 0} = \beta^\lambda_{\rho, i} & \rho = 1, \cdots, 4n
\end{cases}
\]

Here, the set $\{\beta_1, \cdots, \beta_{4n}\}$ forms a basis of $H^0(s(\mathbb{CP}^1), T_Fs)$, and 
\[\beta_{\rho, i} = \big(\beta_{\rho, i}^1(\zeta_i), \cdots, \beta_{\rho, i}^{2n}(\zeta_i)\big)\]
represents the local representation of $\beta_{\rho, i}$ within the local trivialization $U_i \times \mathbb{C}^{2n} \subset T_Fs$. For each $t \in \mathbb{B}_\varepsilon$, the complex submanifold $V_t := \pi^{-1}(t)$ is locally represented by $V_t \cap \mathscr{U}_i = \{z_i^\lambda = \varphi^\lambda_i(\zeta_i, t)\}$ in $\mathcal{Z}$. Define $s_t: \mathbb{CP}^1 \rightarrow \mathcal{Z}$ such that $s_t(\zeta_i) = \{z_i^\lambda = \varphi_i^\lambda(\zeta_i, t), \zeta_i\} \in \mathscr{U}_i$ for $i = 0, 1$. Hence, $V_t := \pi^{-1}(t)$ is the image of the section $s_t$ of $\mathcal{Z}$. Our goal is to demonstrate that $s_t$ has a normal bundle $T_Fs_t \cong \bigoplus^{2n}\mathcal{O}(1)$ for each $t \in \mathbb{B}_\varepsilon^{4n}$, provided that $\varepsilon$ is sufficiently small.

By the Grothendieck lemma \cite{grothendieck1957classification}, $T_Fs_t\cong\bigoplus_{k=1}^{2n}\mathcal{O}(n_k(t))$ as a holomorphic vector bundle. For $\varepsilon$ sufficiently small, we have:

$$
\begin{aligned}
H^0\left(s_t(\mathbb{CP}^1),\bigoplus_{k=1}^{2n}\mathcal{O}\big(n_k(t)\big)\right)&= H^0\left(s(\mathbb{CP}^1),\bigoplus_{k=1}^{2n}\mathcal{O}(1)\right)=4n\\
H^1\left(s_t(\mathbb{CP}^1),\bigoplus_{k=1}^{2n}\mathcal{O}\big(n_k(t)\big)\right)&\leq H^1\left(s(\mathbb{CP}^1),\bigoplus_{k=1}^{2n}\mathcal{O}(1)\right)=0
\end{aligned}
$$
which implies that $n_k(t)\geq-1$, $k=1,\cdots,2n$, and $\sum\limits_{k=1}^{2n}n_k(t)=2n$.

The system of functions 
\[
\left\{\beta_{\rho,i}: U_i \rightarrow \mathbb{C}^{2n}, \zeta_i \mapsto \left(\left. \frac{\partial \varphi_i^\lambda(\zeta_i,t)}{\partial t^\rho} \right|_{t=t_0}\right)_{\lambda=1,\cdots,2n} \right\}_{i=0,1}
\]
defines a holomorphic section of $T_Fs_{t_0}$ for each $\rho=1,\cdots,4n$ when $t_0$ is close to 0, denoted as $\beta_\rho(t_0)$. The set $\big(\beta_1(t_0),\cdots,\beta_{4n}(t_0)\big)$ forms a basis for $H^0\left(s_{t_0}(\mathbb{CP}^1),T_Fs_{t_0}\right)$ that remains holomorphic with respect to $t_0\in \mathbb{B}_\varepsilon^{4n}$. For any fixed $(\gamma^\rho)_{\rho=1}^{4n} \neq 0$ in $\mathbb{C}^{4n}$, the section $\beta:=\sum\limits_{\rho=1}^{4n}\gamma^\rho\beta_\rho$ is a non-trivial holomorphic section of $T_Fs \cong \bigoplus_{k=1}^{2n}\mathcal{O}(1)$ and therefore has at most a simple zero. Let $\beta(t_0) := \sum\limits_{\rho = 1}^{4n} \gamma^\rho \beta_\rho(t_0)$ represent a section of $T_Fs_{t_0}$, specifically $\beta(0) = \beta$. Locally on $\mathscr{U}_i$, $\beta(t_0)|_{\mathscr{U}_i}$ can be expressed as $\big(\beta^1_{i}(\zeta_i,t_0),\ldots,\beta^{2n}_{i}(\zeta_i,t_0)\big)$ for $i = 0, 1$. 

If the holomorphic section $\beta(0)$ never vanishes on $s(\mathbb{CP}^1)$, then neither does $\beta(t_0)$ for $t_0$ in the vicinity of 0. If $\beta$ has a simple zero, we can assume, without loss of generality, that the simple zero point of $\beta$ is $\zeta_0 = 0$, meaning that $\beta^\lambda_{0}(\zeta_0,0) = 0$ for each $\lambda = 1,\ldots,2n$. Specifically, we assume that $\zeta_0 = 0$ is a simple zero point of $\beta^1_{0}(\zeta_0,0)$. Let $\{\zeta_0 \in U_0 : \|\zeta_0\| = 1\}$ be a simple loop around $\zeta_0 = 0$. By Rouché's Theorem, $\beta^1_{0}(\zeta_0,t_0)$ would have exactly one simple zero point in the domain $\{\zeta_0 \in U_0 : \|\zeta_0\| < 1\}$ for each $t_0$ near 0. This implies that the holomorphic section $\beta_0(t_0)$ of $T_Fs_{t_0}$ has at most one zero point in the domain $s_{t_0}\big(\{\zeta_0 \in U_0 : \|\zeta_0\| < 1\}\big) \subset s_{t_0}(U_0)$ for each $t_0$ near 0. Additionally, since $\beta_i^\lambda(\zeta_i,0)$ does not vanish over $U_i \cap \{\zeta_0 \in U_0 : \|\zeta_0\| \leq 1\}$, then $\beta_i^\lambda(\zeta_i,t_0)$ also has no zero points on $U_i \cap \{\zeta_0 \in U_0 : \|\zeta_0\| \leq 1\}$ for $t_0$ near 0. This implies that the holomorphic section $\beta(t_0)$ has at most one simple zero point over $s_{t_0}(\mathbb{CP}^1)$.

Based on the preceding discussion, we can conclude that any non-trivial holomorphic section of $T_Fs_{t_0} \cong \bigoplus_{k=1}^{2n}\mathcal{O}(n_k(t_0))$ has at most one simple zero for $t_0$ in the vicinity of 0, implying that $n_k(t_0) \leq 1$. However, since $\sum_{k=1}^{2n}n_k(t)=2n$, we have established that when $\varepsilon$ is sufficiently small, all $n_k(t_0)$ for $k=1,\ldots,2n$ are equal to 1, i.e.,
\[
T_Fs_t\cong T_Fs\cong\bigoplus_{k=1}^{2n}\mathcal{O}(1) \quad \text{for all} \quad t\in \mathbb{B}^{4n}_\varepsilon.
\]

\end{proof}
In summary, follow the markings above we present the following lemma:

\begin{prop}
\label{first step of section4}
Let $\mathcal{Z}$ be a complex manifold, and let $\sigma$ be the antipodal map over $\mathbb{CP}^1$. We suppose the following: 
\begin{enumerate}
\item $\mathcal{Z}$ is a holomorphic fiber bundle over $\mathbb{CP}^1$, denoted by $p: \mathcal{Z} \rightarrow \mathbb{CP}^1$.
\item $\mathcal{Z}$ admits a holomorphic section $s$ whose normal bundle is $T_Fs \cong \bigoplus^{2n}\mathcal{O}(1)$.
\end{enumerate}

Then $\mathcal{Z}$ admits a family of holomorphic sections that are parameterized by a complex manifold $\mathcal{M}$ of dimension $4n$, where the normal bundle $T_Fs$ is isomorphic to $\bigoplus^{2n}\mathcal{O}(1)$ for each $s \in \mathcal{M}$.
\end{prop}

\begin{fact}
Let $M$ be a smooth manifold and let $G$ be a finite group acting on $M$ by diffeomorphisms. Then the set of fixed points $M^G = \Big\{m \in M : g \cdot m = m, \forall g \in G\Big\}$ is a smooth submanifold of $M$. Moreover, for any $m \in M^G \subset M$, the tangent vector $X_m \in T_mM^G$ if and only if $\mathrm{d}g(X_m) = X_m$ for any $g \in G$. (see, e.g. \cite[Corollary 2.5, p. 309]{bredon1972introduction}.)
\end{fact}

Let $s \in \mathcal{M}$ be holomorphic sections of $p: \mathcal{Z} \rightarrow \mathbb{CP}^1$ such that $T_Fs \cong \bigoplus^{2n}\mathcal{O}(1)$, and we define $\mathfrak{s}:=\overline{\tau \circ s \circ \sigma^{-1}}$, which is also a holomorphic section of $\mathcal{Z}$. We then obtain a diffeomorphism $\iota$ on $\mathcal{M}$
\[
\iota:\mathcal{M}\rightarrow\mathcal{M} \quad s\mapsto\mathfrak{s}
\]
with $\iota^2=\mathrm{id}_\mathcal{M}$ and $s$ is a real section of $\mathcal{M}$ if and only if $\iota(s)=s$. By the fact we mentioned above, $\mathcal{M}$ encompasses a family of real holomorphic sections, denoted as $M$ which is a smooth submanifold of $\mathcal{M}$ since $M$ is the set of fixed points of $\iota$. Suppose $s \in M$ and $X$ is a tangent vector in $T_s\mathcal{M}$, then $X \in T_sM$ if and only if $\mathrm{d}\iota(X) = X$.
 
Let us recall that $r(\tau,s)$ refers to the real structure on $H^0\big(s(\mathbb{CP}^1),T_Fs\big)$ and $j(\tau,s)$ refers to the quaternionic structure on $H^0\big(s(\mathbb{CP}^1),T_Fs\otimes\mathcal{O}(-1)\big)$ induced by $\mathrm{d}\tau:T_Fs\rightarrow\overline{T_Fs}$ according to Definition \ref{induced real structure and quaternionic structure}. It is well-established that $T_s\mathcal{M}\cong H^0\big(s(\mathbb{CP}^1),T_Fs\big)$. Consequently, the real structure $r(\tau,s)$ on $H^0\big(s(\mathbb{CP}^1),T_Fs\big)$ directly induces the subspace $T_sM$ in $T_s\mathcal{M}$, as described below:

\begin{lemma}
For every real holomorphic section $s$ in the manifold $\mathcal{M}$, and for any tangent vector $X\in T_s\mathcal{M}=H^0(s(\mathbb{CP}^1),T_Fs)$, we obtain $\textrm{d}\iota(X)=r(\tau,s)(X)$.
\end{lemma}

\begin{proof}
Let $s_\nu$ be a path in $\mathcal{M}$ for $\nu\in[0,1]$ such that $s_0=s$ and $\left.\frac{\textrm{d}}{\textrm{d}\nu}\right|_{\nu=0}s_\nu=X$.

\begin{align*}
\textrm{d}\iota(X)&=\left.\frac{\textrm{d}}{\textrm{d}\nu}\right|_{\nu=0}\iota(s_\nu)=\lim\limits_{\nu\rightarrow0}\overline{\frac{1}{\nu}(\tau\circ s_\nu\circ\sigma^{-1}-\tau\circ s\circ\sigma^{-1})}\\
&=\overline{\textrm{d}\tau\left(\lim\limits_{\nu\rightarrow0}\frac{1}{\nu}(s_\nu\circ\sigma^{-1}-s\circ\sigma^{-1})\right)}=\overline{\textrm{d}\tau \circ X \circ \sigma^{-1}}
\end{align*}
Furthermore, according to the discussion in proposition \ref{main section2}, we have $r(\tau,s)(X)=\overline{\textrm{d}\tau \circ X \circ \sigma^{-1}}=\textrm{d}\iota(X)$.
\end{proof}

Given a trivialization of the tangent bundle of $T_Fs$ with respect to the standard open cover of the base space $s(\mathbb{CP}^1)$, holomorphic sections of $T_Fs$ can be represented by the functions $\varphi_s(a,b)$, where $a$ and $b$ belong to $\mathbb{C}^{2n}$. Similarly, holomorphic sections of $T_Fs\otimes\mathcal{O}(-1)$ can be represented by the variable $x\in\mathbb{C}^{2n}$. Therefore, according to Corollary \ref{induced real and quaternionic structures}, the real structure $r(\tau,s)\big(\varphi_s(a,b)\big) = \varphi_s\Big(j(\tau,s)(b), -j(\tau,s)(a)\Big)$. With these considerations, we can state the following lemma:

\begin{lemma}
Fix any trivialization of $T_Fs$ with respect to the standard open cover of $s(\mathbb{CP}^1)$. For any $s \in M$, $\varphi_s(a,b)$ is a tangent vector in $T_sM$ if and only if $b = -j(\tau,s)(a)$.
\end{lemma}

For convenience, we will denote $j(s,\tau)$ as $j$ if there is no ambiguity and denote $T_Fs\otimes\mathcal{O}(n)$ as $T_Fs(n)$. The above lemma implies that when $M\neq\emptyset$, it is a smooth submanifold of $\mathcal{M}$ with dimension $4n$ because the tangent space $T_sM$ of any $s \in M$ is a $2n$-dimensional $\mathbb{R}$-linear subspace of $T_s\mathcal{M}$. It is worth noting that $T_sM$ is not a $\mathbb{C}$-linear space of $T_s\mathcal{M}$, and if we want $M$ to be a complex manifold, we must first give a complex structure to $T_sM$. Let us do this as follows:

If $\varphi_s(a,-j(a))$ is a tangent vector in $T_sM$ which vanishes at $\zeta=\zeta_0$, this means that $a=\zeta_0\cdot j(a)$. Therefore, we have 
\[
j(a)=j(\zeta_0\cdot j(a))=-\bar\zeta_0\cdot a=-\|\zeta_0\|^2j(a)
\]
which implies that $j(a)=0$, and thus $a=0$ since $j$ is a bijective map. In summary, if we regard a tangent vector $X$ in $T_sM$ as a holomorphic section in $T_Fs$, then if $X$ vanishes at one point on $s(\mathbb{CP}^1)$, it has to vanish everywhere on $s(\mathbb{CP}^1)$, meaning that $X=0$. With this, we can state the following lemma:

\begin{lemma}
\label{induced complex structure}
For each fiber $\mathcal{Z}_\zeta:=p^{-1}(\zeta)$ of $p:\mathcal{Z}\rightarrow\mathbb{CP}^1$ and any real section $s\in M$, there exists a neighborhood $U$ of $s$ in $M$ such that $U$ is a family of real sections in $\mathcal{Z}$ that intersect $p^{-1}(\zeta)$ at distinct points, provided that $U$ is sufficiently small.
\end{lemma}

For this lemma, $M$ can be locally identified with any one of the fibers in $\{p^{-1}(\zeta)|\zeta\in\mathbb{CP}^1\}$. Therefore, $M$ inherits a complex structure $\mathbf{I}_\zeta$ from the fiber $p^{-1}(\zeta)$ of $\mathcal{Z}$. Thus, we obtain a family of complex structures on $M$ parameterized by $\mathbb{CP}^1$. Furthermore, $\textrm{id}_M:(M,\mathbf{I}_\zeta)\rightarrow(M,\mathbf{I}_{\sigma(\zeta)})$ is an antiholomorphic map. We can also introduce another family of complex structures parameterized by $\mathbb{CP}^1$ on $M$ as follows:

\begin{lemma}
For every element $s \in M$, the tangent space $T_sM$ constitutes a real linear subspace within $T_s\mathcal{M}$. Furthermore, there exists a set of $\mathbb{R}$-linear isomorphisms connecting $T_sM$ with $H^0\big(s(\mathbb{CP}^1),T_Fs(-1)\big)$.
\end{lemma}

\begin{proof}
After fixing a trivialization of $T_Fs$ with respect to the standard open cover of $\mathbb{CP}^1$, we can define a $\mathbb{R}$-linear map, denoted as $\Phi_{\alpha,\beta}: T_sM \rightarrow H^0\big(s(\mathbb{CP}^1),T_F(-1)\big)$, for constants $(\alpha,\beta) \in \mathbb{C}^2 \setminus \{0\}$, as follows:
\[
\Phi_{\alpha,\beta}:\varphi_s\Big(a,-j(\tau,s)(a)\Big) \mapsto \alpha\cdot a - \beta\cdot j(\tau,s)(a) \in \mathbb{C}^{2n}
\]
Firstly, the map $\Phi_{\alpha,\beta}$ is an isomorphism. This is demonstrated by the fact that if $\Phi_{\alpha,\beta}$ sends $\varphi_s\big(a,-j(\tau,s)(a)\big)$ to zero, then it must be that $\alpha\cdot a = \beta\cdot j(\tau,s)(a)$. Applying $j(\tau,s)$ to both sides of the equation yields $\bar\alpha\cdot j(\tau,s)(a) = -\bar\beta\cdot a$, leading to:
\[
|\alpha|^2\cdot a=\bar \alpha\beta\cdot j(\tau,s)(a)=-|\beta|^2\cdot a
\]
which implies that $a=0$, i.e., $\varphi_s\big(a,-j(\tau,s)(a)\big)$ is the zero section. Moreover, both $T_sM$ and $H^0\big(s(\mathbb{CP}^1), T_F(-1)\big)$ have the same real dimension of $4n$. Secondly, the map $\Phi_{\alpha,\beta}$ is independent of the choice of trivialization on $T_Fs$. 
\end{proof}

\begin{remark}
Up to now, we have established an equivalence between $T_sM$ and $H^0\big(s(\mathbb{CP}^1), T_F(-1)\big)$ induced from the map $\Phi_{\alpha,\beta}$. Furthermore, $\Phi_{\alpha,\beta}$ endows $T_sM$ with a complex structure derived from the complex vector space $H^0\big(s(\mathbb{CP}^1), T_F(-1)\big)$. This is a family of complex structures on $M$ parameterized by $[\alpha:\beta]\in\mathbb{CP}^1$.
\end{remark}

Lemma \ref{induced complex structure} indicates that $T_sM$ inherits a complex structure from each fiber $\mathcal{Z}_\zeta$. We will clarify that the complex structures on $T_sM$ inherited from $\mathcal{Z}_\zeta$ correspond precisely to those complex structures on $T_sM$ induced from $H^0\big(s(\mathbb{CP}^1), T_F(-1)\big)$ by $\Phi_{\alpha,\beta}$. To elucidate this fact, we first need to establish a trivialization of $T_Fs$.

Consider the standard open cover $U_0 \cup U_1$ of $\mathbb{CP}^1$, where each $U_i$ is isomorphic to $\mathbb{C}$ with a local coordinate $\zeta_i$ for $i = 0,1$. There exists a corresponding open cover $\mathscr{U}_0 \cup \mathscr{U}_1$ of $s(\mathbb{CP}^1)$ within $\mathcal{Z}$. Each $\mathscr{U}_i$ can be represented as $\mathbb{B}_i \times U_i$, with $\mathbb{B}_i$ being a neighborhood of zero in $\mathbb{C}^{2n}$. Without loss of generality, we can assume that $s(\mathbb{CP}^1) \cap \mathscr{U}_i$ corresponds to zero in $\mathbb{B}_i$. Since $T_Fs\cong\bigoplus^{2n}\mathcal{O}(1)$, we claim that there exist coordinates $z_i:=(z^\lambda_i)_{\lambda=1,\ldots,2n}$ on $\mathbb{B}_i$ such that 
\[
\left.\frac{\partial}{\partial z_0^\lambda}\right|_{z_0=0}=\zeta_1\cdot\left.\frac{\partial}{\partial z_1^\lambda}\right|_{z_1=0} \text{for all}\quad \lambda=1,\ldots,2n. 
\]
We will demonstrate this fact later. Consequently, we have the following trivialization of $T_Fs$:

\[
\begin{cases}
U_0 \times \mathbb{C}^{2n} \rightarrow T_Fs &
(\zeta_0, a_0^1, \cdots, a_0^{2n}) \mapsto \sum\limits_{\lambda=1}^{2n} a_0^\lambda \left.\frac{\partial}{\partial z_0^\lambda}\right|_{s(\zeta_0)}\\
U_1 \times \mathbb{C}^{2n} \rightarrow T_Fs &
(\zeta_1, a_1^1, \cdots, a_1^{2n}) \mapsto \sum\limits_{\lambda=1}^{2n} a_1^\lambda \left.\frac{\partial}{\partial z_1^\lambda}\right|_{s(\zeta_1)}
\end{cases}
\]

This particular trivialization of $T_Fs$ will be used throughout the remainder of this paper. Consider a tangent vector $X_s^a:=\varphi_s(a,-j(\tau,s))$ in $T_sM$ for a real section $s$ under the trivialization we defined above. In detail, we can write $X_s^a$ as:
\begin{align}
\begin{cases}
\label{bundle trivialization}
X_s^a|_{s(\zeta_0)}=\sum\limits_{\lambda=1}^{2n}(a^\lambda-\zeta_0\big(j(\tau,s)(a)\big)^\lambda)\left.\frac{\partial}{\partial z_0^\lambda}\right|_{s(\zeta_0)}\\
X_s^a|_{s(\zeta_1)}=\sum\limits_{\lambda=1}^{2n}(\zeta_1a^\lambda-\big(j(\tau,s)(a)\big)^\lambda)\left.\frac{\partial}{\partial z_1^\lambda}\right|_{s(\zeta_1)}
\end{cases}
\end{align}

Let us consider the complex structure $\mathbf{I}_{\zeta_0}$ on $T_sM$, which is inherited from $\mathcal{Z}_{\zeta_0}$ for some $\zeta_0 \in U_0$. This complex structure sends $X_s^a|_{s(\zeta_0)}$ to $\sqrt{-1}\cdot X_s^a|_{s(\zeta_0)}$, and in detail:

\[
\sum\limits_{\lambda=1}^{2n}\Big(a^\lambda - \zeta_0\big(j(\tau,s)(a)\big)^\lambda\Big)\left.\frac{\partial}{\partial z_0^\lambda}\right|_{s(\zeta_0)}\mapsto \sqrt{-1} \cdot \left(\sum\limits_{\lambda=1}^{2n}\Big(a^\lambda - \zeta_0\big(j(\tau,s)(a)\big)^\lambda\Big)\left.\frac{\partial}{\partial z_0^\lambda}\right|_{s(\zeta_0)}\right).
\]
Therefore, $\mathbf{I}_{\zeta_0}$ maps $X_s^a$ to $X_s^b$, where $b\in\mathbb{C}^{2n}$ satisfies the equation:
\[
b - \zeta_0 \cdot j(\tau,s)(b) = \sqrt{-1} \cdot \big(a - \zeta_0 \cdot j(\tau,s)(a)\big).
\]
In fact, we can solve this equation and obtain:
\[
b = \frac{\sqrt{-1}(1-|\zeta_0|^2)}{1+|\zeta_0|^2}a - \frac{2\sqrt{-1}\zeta_0}{1+|\zeta_0|^2}j(a).
\]

On the other hand, let 
\[
\alpha=\frac{1-|\zeta_0|^2}{1+|\zeta_0|^2}\text{ and }\beta=\frac{2\zeta_0}{1+|\zeta_0|^2}.
\] 
then, the complex structure induced by $\Phi_{\alpha,\beta}$ from $H^0\big(s(\mathbb{CP}^1),T_Fs(-1)\big)$ is just the complex structure inherited from $\mathcal{Z}_{\zeta_0}$. Up to this point, we have successfully obtained a hypercomplex structure on $M$, as discussed in the following proposition. Notably, Joyce has also explored hypercomplex structures in his book \cite[Section 7.5]{joyce2000compact}.

\begin{theorem}
\label{hypercomplex}
Consider a complex manifold $\mathcal{Z}$ equipped with an antipodal map $\sigma$ over the projective line $\mathbb{CP}^1$. We impose the following conditions for $\mathcal{Z}$ to facilitate our study:
\begin{enumerate}
    \item The manifold $\mathcal{Z}$ must be structured as a holomorphic fiber bundle over $\mathbb{CP}^1$, formally represented by the projection $p: \mathcal{Z} \rightarrow \mathbb{CP}^1$.
    \item It is necessary for $\mathcal{Z}$ to possess a real structure $\tau$, which aligns with the antipodal map $\sigma$, ensuring compatibility across the manifold.
    \item The existence of a holomorphic section $s$ is crucial, where the section's normal bundle aligns with $T_Fs \cong \bigoplus^{2n}\mathcal{O}(1)$, highlighting the manifold's dimensional complexity.
\end{enumerate}
Given these stipulations, let $M$ denote the collection of real holomorphic sections within $\mathcal{M}$. Should $M$ be non-empty, it inherently forms a smooth submanifold of $\mathcal{M}$, further endowed with a hypercomplex structure.
\end{theorem}

\begin{proof}
Let $s$ be a real section of $\mathcal{Z}$, and with respect to the trivialization of $T_Fs$ mentioned above, we can write any tangent vector in $T_sM$ in the form $\varphi_s(a,-j(a))$ for some $a$ in $\mathbb{C}^{2n}$. Let $\mathbf{I}_{\alpha,\beta}$ be a family of complex structures on $M$ induced by $\Phi_{\alpha,\beta}:\varphi_s\big(a,-j(a)\big) \mapsto \alpha\cdot a-\beta\cdot j(a) $, in more detail, $\mathbf{I}_{\alpha,\beta}$ sends $\varphi_s(a,-j(a))$ to $\varphi_s(b,-j(b))$ for some $b$ in $\mathbb{C}^{2n}$ such that
\[
\alpha\cdot b-\beta\cdot j(b)=\sqrt{-1}\cdot\big(\alpha\cdot a-\beta\cdot j(a)\big)
\]
We can solve this equation for certain pairs $(\alpha,\beta)$ and obtain
\[
\begin{cases}
\mathbf{I}_{1,0}:\varphi_s\big(a,-j(a)\big)\mapsto\varphi_s\big(\sqrt{-1}\cdot a,-j(\sqrt{-1}\cdot a)\big)\\
\mathbf{I}_{1,i}:\varphi_s\big(a,-j(a)\big)\mapsto\varphi_s\big(j(a),a\big)
\end{cases}
\]
Denote $\mathbf{I}:=\mathbf{I}_{1,0}$, $\mathbf{J}:=\mathbf{I}_{1,i}$, and $\mathbf{K}:=\mathbf{I}\circ\mathbf{J}$. We note that $\mathbf{I}$ is the complex structure on $M$ inherited from the fiber $\mathcal{Z}_0$ and $\mathbf{J}$ is the complex structure on $M$ inherited from the fiber $\mathcal{Z}_{\sqrt{-1}(-1\pm\sqrt{2})}$. Then, $\mathbf{I},\mathbf{J},\mathbf{K}$ are three complex structures on $M$ and can check that $\mathbf{K}=\mathbf{I}\circ\mathbf{J}=-\mathbf{J}\circ\mathbf{I}$ Then $(M,\mathbf{I},\mathbf{J},\mathbf{K})$ is a hypercomplex manifold.
\end{proof}
To complete the proof of Theorem \ref{hypercomplex}, it is essential to calculate the local charts on $\mathcal{Z}$ in the vicinity of point $s$.

\begin{lemma}
Coordinates $z_i := (z^\lambda_i)_{\lambda=1,\cdots,2n}$ can be selected on $\mathbb{B}_i$ such that $\left.\frac{\partial}{\partial z_0^\lambda}\right|_{z_0=0} = \zeta_1\cdot\left.\frac{\partial}{\partial z_1^\lambda}\right|_{z_1=0}$ holds for every $\lambda=1,\cdots,2n$.
\end{lemma}

\begin{proof}
Select any local coordinates $(z_i^1,\cdots, z_i^{2n},\zeta_i)$ on $\mathcal{Z}$ where $\mathscr{U}_i$ serves as the local chart for $i = 0,1$. Define $z_i=(z_i^1,\cdots,z_i^{2n})$ for $i = 0,1$ and, without loss of generality, assume $s(\mathbb{CP}^1) \cap \mathscr{U}_i = \{z_i = 0\}$ for each $i$. The transition functions between $\mathscr{U}_0$ and $\mathscr{U}_1$ are:

\[
\begin{cases}
\zeta_0 \cdot \zeta_1 = 1 & \text{on } U_0 \cap U_1 \cong \mathbb{C}^*\\
z_1^\lambda = f^\lambda(z_0,\zeta_0) & \text{for } \lambda = 1,\cdots, 2n
\end{cases}
\]

Observe that $f^\lambda(0,\zeta_0)=0$. Define $\mathscr{V}_i = \mathscr{U}_i \cap s(\mathbb{CP}^1)$ for $i = 0,1$. The transition matrix of $T_Fs$ from $\mathscr{V}_0$ to $\mathscr{V}_1$ is:

\[
F := \left.\left(
\begin{array}{ccc}
\frac{\partial f^1}{\partial z_0^1} & \cdots & \frac{\partial f^1}{\partial z_0^{2n}}\\
\vdots & \vdots & \vdots\\
\frac{\partial f^{2n}}{\partial z_0^1} & \cdots & \frac{\partial f^{2n}}{\partial z_0^{2n}}
\end{array}
\right)\right|_{z_0=0}
\]
Since $T_Fs\cong\bigoplus^{2n}\mathcal{O}(1)$, there exist matrix-valued holomorphic functions $A_0:\mathscr{V}_0\rightarrow GL(2n,\mathbb{C})$ and $A_1:\mathscr{V}_1\rightarrow GL(2n,\mathbb{C})$ such that:
\[
A_1^{-1}(\zeta_0^{-1})\cdot F\cdot A_0(\zeta_0)=\left(
\begin{array}{ccc}
\zeta_0^{-1} & &\\
& \ddots &\\
& & \zeta_0^{-1}
\end{array}
\right)
\]
Let $(w_i,\eta_i) := (A_i(\zeta_i) \cdot z_i,\zeta_i)$ be new coordinates on $\mathscr{U}_i$ for $i=0,1$. The corresponding transition functions between $\mathscr{U}_0$ and $\mathscr{U}_1$ are:
\[
\begin{cases}
\eta_0 \cdot \eta_1 = 1 \\
w_1^\lambda = g^\lambda(w_0,\eta_0) := \left(A^{-1}(\eta_1)\right)_\mu^\lambda\cdot f^\mu(A_0(\eta_0)\cdot w_0,\eta_0)
\end{cases}
\]
The transition matrix $T_Fs$ from $\mathscr{V}_0$ to $\mathscr{V}_1$ is:
\[
G := \left.\left(\begin{array}{ccc}
\frac{\partial g^1}{\partial w_0^1} & \cdots & \frac{\partial g^1}{\partial w_0^{2n}}\\
\vdots & \vdots & \vdots\\
\frac{\partial g^{2n}}{\partial w_0^1} & \cdots & \frac{\partial g^{2n}}{\partial w_0^{2n}}
\end{array}\right)\right|_{w_0=0}=A_1^{-1}(\zeta_0^{-1})\cdot F\cdot A_0(\zeta_0)
\]
The refined transition functions between $\mathscr{U}_0$ and $\mathscr{U}_1$ are thus:

\[
\begin{cases}
\eta_0 \cdot \eta_1 = 1 & \text{where } \eta_i\in \mathscr{V}_0\cap\mathscr{V}_1\\
w_1^\lambda = \eta_0^{-1}\cdot w_0^\lambda + h(w_0,\eta_0)
\end{cases}
\]
Here, $h^\lambda(0,\eta_0)=\left.\frac{\partial h^\lambda}{\partial w_0^\mu}\right|_{w_0=0}=0$ for each $\lambda,\mu=1,\cdots, 2n$. Consequently, $\left.\frac{\partial}{\partial w_0^\lambda}\right|_{w_0=0}=\eta_1\cdot\left.\frac{\partial}{\partial w_1^\lambda}\right|_{w_1=0}$ is valid for every $\lambda=1,\cdots,2n$.
\end{proof}

Finally, we establish the central theorem of this study, initially put forth by Hitchin, Karlhede, Lindstr{\"o}m, and Ro{\v{c}}ek in their work \cite{hitchin1987hyperkahler}.

\begin{theorem}
\label{main section4}
Let $\mathcal{Z}$ be a complex manifold and $\sigma$ be the antipodal map over $\mathbb{CP}^1$. Suppose that:
\begin{enumerate}
    \item $\mathcal{Z}$ is a holomorphic fiber bundle over $\mathbb{CP}^1$, denoted by $p: \mathcal{Z} \rightarrow \mathbb{CP}^1$.
    \item There exists a real structure $\tau$ on $\mathcal{Z}$ that is compatible with $\sigma$.
    \item $\mathcal{Z}$ admits a family $\mathcal{M}$ of holomorphic sections, where the normal bundle $T_Fs$ is isomorphic to $\bigoplus^{2n}\mathcal{O}(1)$ for each $s \in \mathcal{M}$.
    \item There exists a holomorphic section $\omega$ of $\Lambda^2{(N\mathcal{Z})}^*\otimes p^*\mathcal{O}(2)$ such that:
    \begin{enumerate}
        \item $\omega$ defines a symplectic form on each fiber of $p: \mathcal{Z} \rightarrow \mathbb{CP}^1$.
        \item For each real section $s \in \mathcal{M}$, $\omega(\cdot, j(\tau, s)\cdot)$ is a negative definite quadratic form on $H^0\left(s(\mathbb{CP}^1), T_Fs\otimes\mathcal{O}(-1)\right) \cong \mathbb{C}^{2n}$.
    \end{enumerate}
\end{enumerate}
Let $M$ be the set of real holomorphic sections in $\mathcal{M}$. Under these conditions, if $M\neq\emptyset$, then it is a smooth submanifold of $\mathcal{M}$. Additionally, there exists a hyperk\"ahler structure on $M$.
\end{theorem}

\begin{remark}
$\text{(4-b)}$ states that for any $x \in H^0\left(s(\mathbb{CP}^1),T_Fs\otimes\mathcal{O}(-1)\right)$, we always have $\omega\big(x,j(\tau,s)(x)\big) \leq 0$, and $\omega\big(x,j(\tau,s)(x)\big) = 0$ if and only if $x = 0$.
\end{remark}

\begin{proof}

For any real section $s\in\mathcal{M}$, the restriction of $\omega$ on $s(\mathbb{CP}^1)$ denoted as $\left.\omega\right|_{s(\mathbb{CP}^1)}$ is a holomorphic section of $\Lambda^2 \left(T_Fs\right)^*\otimes\mathcal{O}(2)$. It can be considered as a 2-form with values in $\mathcal{O}(2)$. This indicates that $\omega$ is a nondegenerate skew-symmetric 2-form on $H^0\left(s(\mathbb{CP}^1),T_Fs(-1)\right)$. 

For the sake of clarity, our analysis will concentrate exclusively on the complex structure in $T_sM$ derived from $\mathcal{Z}_0$. More precisely, we investigate the complex structure engendered by the mapping $\Phi: T_sM \rightarrow H^0\big(s(\mathbb{CP}^1), T_Fs(-1)\big)$, which is defined such that $\varphi_s\big(a, -j(a)\big) \mapsto a \in \mathbb{C}^{2n}$. In this framework, the negative definite quadratic form $\omega$, as referenced in condition (4), is considered as a 2-form on $T_sM$. We denote $\omega$'s restriction on $\mathcal{Z}_{\zeta_0}$ for any $\zeta_0 \in U_0$ as $\omega_{\zeta_0}$. Assuming $s$ to be a real section and considering $a, b \in \mathbb{C}^{2n}$ as sections of $T_Fs(-1)$ under the specified trivialization, $\omega(a,b)$ is thus rendered a holomorphic function on $\mathbb{CP}^1$, implying its constancy. Moreover, $\left.\omega(a,b)\right|_{\mathcal{Z}_{\zeta_0}}$ equates to $\omega_{\mathcal{Z}_{\zeta_0}}\left(a|_{s(\zeta_0)},b|_{s(\zeta_0)}\right)$. Therefore, it follows that $\omega(a,b)=\omega_{\mathcal{Z}_{\zeta_0}}\left(a|_{s(\zeta_0)},b|_{s(\zeta_0)}\right)$ for every $\zeta_0\in U_0\subset\mathbb{CP}^1$.

The quaternionic structures $j$ on $H^0(s(\mathbb{CP}^1),T_F(-1))$ induce quaternionic structures $\mathscr{J}:=\Phi\circ j\circ\Phi^{-1}$ on $T_sM$. We can represent the quaternionic structure $\mathscr{J}$ on $T_sM$ as follows: using the trivialization of $T_Fs$ outlined in equation \ref{bundle trivialization}, we have
\[
\mathscr{J}: T_sM \rightarrow T_sM, \quad \varphi_s\big(a,-j(a)\big) \mapsto \varphi_s\big(j(a),a\big).
\]
For any $s \in M$, the metric $g: T_sM \times T_sM \rightarrow \mathbb{R}$ is defined by 
\[
\Big(\varphi_s\big(a,-j(a)\big),\varphi_s\big(b,-j(b)\big)\Big) \mapsto -\omega\big(a,j(b)\big)-\omega\big(b,j(a)\big),
\]
with the well-definedness of the Riemannian metric $g$ on $M$ being confirmed by condition (4-b).

Let $X_s^a:=\varphi_s\big(a,-j(a)\big)$ denote a tangent vector in $T_sM$ and a normal vector field on $s(\mathbb{CP}^1)$. Choose $\zeta_0\in U_0\subset {\mathbb {CP}}^1$ and consider a local isomorphism $f:U\subset M\rightarrow U^\prime\subset\mathcal{Z}_{\zeta_0}:=p^{-1}(\zeta_0)$. For every $a\in\mathbb{C}^{2n}$ and $s\in U$, utilizing the trivialization of $T_Fs$ as described in equation \ref{bundle trivialization}, we obtain:
\[
{\rm d}f(X_s^a)=\left. X_s^a\right|_{s(\mathcal{\zeta}_0)}=\sum_{\lambda=1}^{2n}\Big(a-j(a)\zeta_0\Big)^\lambda\left.\frac{\partial}{\partial z_0^\lambda}\right|_{s(\zeta_0)}.
\]
Next, we define a 2-form $\psi_{\zeta_0}$ on $U$ by $\psi_{\zeta_0}:T_sM\times T_sM\rightarrow \mathbb{C}$, where $\psi_{\zeta_0}(X_s^a,X_s^b):=\omega\big(a-j(a)\zeta_0,b-j(b)\zeta_0\big)$. Here, $a-j(a)\zeta_0$ and $b-j(b)\zeta_0$ are viewed as sections of $T_Fs(-1)$, and importantly, for any $\zeta_0\in U_0$, 
\[
\omega\Big(a-j(a)\zeta_0,b-j(b)\zeta_0\Big)=\omega_{\mathcal{Z}_{\zeta_0}}\left(\sum_{\lambda=1}^{2n}\big(a-j(a)\zeta_0\big)^\lambda\left.\frac{\partial}{\partial z_0^\lambda}\right|_{s(\zeta_0)},\sum_{\lambda=1}^{2n}\big(b-j(b)\zeta_0\big)^\lambda\left.\frac{\partial}{\partial z_0^\lambda}\right|_{s(\zeta_0)}\right)
\]
Thus, $\psi_{\zeta_0}=f^*\big(\omega|_{U^\prime}\big)$ constitutes a closed form on $U$, given that $\omega_{\mathcal{Z}_{\zeta_0}}$ is a holomorphic symplectic form on $\mathcal{Z}_{\zeta_0}$, in accordance with condition (4-a).

A complex structure is defined on $T_sM$ for each $s\in M$ as follows:
\[
\begin{cases}
&\mathbf{I}:\varphi_s\big(a,-j(a)\big)\mapsto\varphi_s\big(\sqrt{-1}\cdot a,-j(\sqrt{-1}\cdot a)\big), \\
&\mathbf{J}:\varphi_s\big(a,-j(a)\big)\mapsto\varphi_s\big(j(a),a\big), \\
&\mathbf{K}:\varphi_s\big(a,-j(a)\big)\mapsto\varphi_s\big(\sqrt{-1}\cdot j(a),-\sqrt{-1}\cdot a\big).
\end{cases}
\]
Here, $\mathbf{I}^2 = \mathbf{J}^2 = \mathbf{K}^2 = -\mathrm{id}$ and $\mathbf{I} \circ \mathbf{J} = \mathbf{K}$. It suffices to verify that $g(\mathbf{I}\cdot, \mathbf{I}\cdot) = g(\mathbf{J}\cdot, \mathbf{J}\cdot) = g(\mathbf{K}\cdot, \mathbf{K}\cdot) = g(\cdot, \cdot)$ and that $g(\mathbf{I}\cdot,\cdot)$, $g(\mathbf{J}\cdot,\cdot)$, $g(\mathbf{K}\cdot,\cdot)$ are closed forms.

\begin{lemma}
$g(\mathbf{I}\cdot, \mathbf{I}\cdot) = g(\mathbf{J}\cdot, \mathbf{J}\cdot) = g(\mathbf{K}\cdot, \mathbf{K}\cdot) = g(\cdot, \cdot)$.
\end{lemma}

\begin{proof} 
Given any $X_s^a := \varphi_s\big(a, -j(a)\big)$ and $X_s^b := \varphi_s\big(b, -j(b)\big)$ in $T_sM$, we find that
\begin{align*}
g\big(\mathbf{I}X_s^a, \mathbf{I}X_s^b\big) &= g\big(\varphi(\sqrt{-1} \cdot a, \sqrt{-1} \cdot j(a)), \varphi(\sqrt{-1} \cdot b, \sqrt{-1} \cdot j(b))\big) \\
&= -\omega\big(\sqrt{-1} \cdot a, j(\sqrt{-1} \cdot b)\big) - \omega\big(\sqrt{-1} \cdot b, j(\sqrt{-1} \cdot a)\big) \\
&= \omega\big(a, j(b)\big) - \omega\big(b, j(a)\big) \\
&= g(X_s^a, X_s^b).
\end{align*}

Similarly, $g(\mathbf{J}X_s^a, \mathbf{J}X_s^b) = g(X_s^a, X_s^b)$ and $g(\mathbf{K}X_s^a, \mathbf{K}X_s^b) = g(X_s^a, X_s^b)$. 
\end{proof}

\begin{lemma}
$g(\mathbf{I}\cdot,\cdot)$, $g(\mathbf{J}\cdot,\cdot)$, and $g(\mathbf{K}\cdot,\cdot)$ are closed 2-forms on $M$.
\end{lemma}

\begin{proof} 
Consider any $s_0 \in M$ and $\zeta_0 \in \mathbb{CP}^1$. Here, $\psi_{\zeta_0}$ is a closed 2-form on a neighborhood $U$ of $s_0$ in $M$, defined by
\[
\psi_{\zeta_0} (X_s^a, X_s^b) := \omega\Big(a - j(\tau, s)(a) \zeta_0, b - j(\tau, s)(b) \zeta_0\Big)
\] 
for all $s \in U$. Therefore,
\begin{align*}
g(\mathbf{I}X_s^a, X_s^b) &= -\omega\Big(\sqrt{-1}\cdot a, j(\tau,s)(b)\Big) - \omega\Big(b, j(\tau,s)(\sqrt{-1}\cdot a)\Big)\\
& = -\sqrt{-1}\Big(\omega\big(a, j(\tau,s)(b)\big) + \omega\big(j(\tau,s)(a), b\big)\Big)\\
& = -\frac{\sqrt{-1}}{2}\big(\psi_{-1} - \psi_1\big)\Big(X_s^a, X_s^b\Big).
\end{align*}
This shows that $g(\mathbf{I}\cdot, \cdot)$ coincides with $-\frac{\sqrt{-1}}{2}(\psi_{-1} - \psi_1)$, a closed 2-form in $U$. Consequently, as closed forms are defined locally, we conclude that $g(\mathbf{I}\cdot, \cdot)$ is a closed 2-form. Similar reasoning applies to $g(\mathbf{J}\cdot, \cdot)$ and $g(\mathbf{K}\cdot, \cdot)$, establishing them as closed 2-forms. 
\end{proof}
With this, we conclude the proof of the theorem \ref{main section4} : $(M,g,\mathbf{I},\mathbf{J},\mathbf{K})$ is a hyperkähler manifold.

\end{proof}


\section{the twistor space of $\mathbb{C}^{2n}$}

\begin{prop}
Considering $\mathbb{C}^{2n}$ as a hyperkähler manifold, the twistor space $\mathcal{Z}$ constructed from $\mathbb{C}^{2n}$ is biholomorphically equivalent to the vector bundle $\bigoplus^{2n}\mathcal{O}(1)$. Conversely, starting from the holomorphic vector bundle $\bigoplus^{2n}\mathcal{O}(1)$ equipped with a specific real structure, a deformation space $\mathcal{M}$ of the zero section $0_s$ in the bundle can be established. Furthermore, $\mathcal{M}$ is isomorphic to $\mathbb{C}^{4n}=(a^1,a^2,b^1,b^2)$, with $a^1,a^2,b^1,$ and $b^2$ being elements of $\mathbb{C}^n$. The hyperkähler manifold $M$, derived via this approach, can be considered as a smooth submanifold of $\mathcal{M}$, represented by $M\subset\mathcal{M}$ as $(a^1,a^2,-i\bar a^2,i\bar a^1)$.
\end{prop}
\begin{remark}
In subsequent analysis, we will focus on the case where $n=1$ for simplicity. It is important to note that scenarios with $n>1$ can be conceptualized as iterative applications of the $n=1$ case.
\end{remark}

Consider the complex vector space $\mathbb{C}^2$ with complex coordinates $(w^1, w^2)$. Let $M\cong\mathbb{C}^2$ be a smooth manifold isomorphic to $\mathbb{R}^4$. A natural global coordinate system exists on $M$, denoted by $\big(\Re w^1, \Im w^1, \Re w^2, \Im w^2\big)$, and alternatively represented as $(w^1, w^2)$. The complexification of the real tangent bundle $TM$ yields a trivial complex bundle $T_{\mathbb{C}}M$ of rank four, framed globally by $\Big(\frac{\partial}{\partial w^1},\frac{\partial}{\partial w^2}, \frac{\partial}{\partial \bar{w}^1},\frac{\partial}{\partial \bar{w}^2}\Big)$. A complex structure on $M$, defined as a bundle map $\mathbf{I}:TM\rightarrow TM$ satisfying $\mathbf{I}^2=-\mathrm{id}$, extends to $T_{\mathbb{C}}M$ as a $\mathbb{C}$-linear map $\mathbf{I}:T_{\mathbb{C}}M\rightarrow T_{\mathbb{C}}M$ with $\mathbf{I}^2 = -\mathrm{id}$. Employing the frame on $T_{\mathbb{C}}M$, we can represent  three complex structures $\mathbf{I},\mathbf{J},\mathbf{K}$ as matrices:

\[
\mathbf{I}\sim\begin{pmatrix}
i &  &  &  \\
 & i &  &  \\
 &  & -i &  \\
 &  &  & -i
\end{pmatrix};\quad
\mathbf{J}\sim\begin{pmatrix}
 &  &  & 1 \\
 &  & -1 &  \\
 & 1 &  &  \\
-1 &  &  & 
\end{pmatrix};\quad
\mathbf{K}\sim\begin{pmatrix}
 &  &  & i \\
 &  & -i &  \\
 & -i &  &  \\
i &  &  & 
\end{pmatrix}.
\]

Now denote the standard Euclidean metric $g$ on $M$ by: 
\[
g=\frac{1}{2}\Big(\textrm{d}w^1\otimes \textrm{d}\bar w^1+\textrm{d}\bar w^1\otimes \textrm{d}w^1+\textrm{d}w^2\otimes \textrm{d}\bar w^2+\textrm{d}\bar w^2\otimes \textrm{d}w^2\Big)
\]
We can verify that $(M,g,\mathbf{I},\mathbf{J},\mathbf{K})$ is a hyperk\"ahler manifold.

\begin{lemma}
The twistor space $\mathcal{Z}$ is biholomorphically equivalent to the vector bundle $\bigoplus^2\mathcal{O}(1)$ over $\mathbb{CP}^1$.
\end{lemma}

\begin{remark}
According to Hitchin-Karlhede-Lindstr{\"o}m-Ro{\v{c}}ek's theorem \cite{hitchin1987hyperkahler}, we can readily establish the aforementioned lemma. Indeed, it is feasible to explicitly construct an isomorphism $\phi:\bigoplus^2\mathcal{O}(1)\xrightarrow{\sim}\mathcal{Z}$.
\end{remark}

\begin{proof}
First, we consider the standard open cover of $\mathbb{CP}^1$, denoted as $U_0 \cup U_1$, with coordinates $\zeta_0$ and $\zeta_1$. A local trivialization for $\bigoplus^2\mathcal{O}(1)$ is provided relative to this cover on $\mathbb{CP}^1$. The bundle is represented as $\bigoplus^2\mathcal{O}(1) = (U_0 \times \mathbb{C}^2) \cup (U_1 \times \mathbb{C}^2) / \sim$, with coordinates $(\zeta_0, z_0^1, z_0^2)$ and $(\zeta_1, z_1^1, z_1^2)$ on them. The equivalence of $(\zeta_0, z_0^1, z_0^2) \in U_0 \times \mathbb{C}^2$ and $(\zeta_1, z_1^1, z_1^2) \in U_1 \times \mathbb{C}^2$ holds if $\zeta_0 \cdot \zeta_1 = 1$, $z_0^1 = \zeta_0 \cdot z_1^1$, and $z_0^2 = \zeta_0 \cdot z_1^2$. For the twistor space $\mathcal{Z}$, constructed from $M$ and diffeomorphic to $M \times \mathbb{CP}^1$, an open cover is defined as $(U_0 \times M) \cup (U_1 \times M) / \sim$, with coordinates $(\zeta_i, w_i^1, w_i^2)$ on $U_i \times M$ for $i=0, 1$. The points $(\zeta_0, w_0^1, w_0^2) \in U_0 \times M$ and $(\zeta_1, w_1^1, w_1^2) \in U_1 \times M$ are equivalent in $\mathcal{Z}$ if $\zeta_0 \cdot \zeta_1 = 1$, $w_0^1 = w_1^1$, and $w_0^2 = w_1^2$.

The complex structure $\underline{\mathbf{I}}$ on $\mathcal{Z}$ is defined as follows: The restriction of $\underline{\mathbf{I}}$ to $T_{(\cdot, \zeta_0)}\mathcal{Z}$ is given by $\mathbf{I}_{\zeta_0} \oplus \mathbf{I}_0$, where $\mathbf{I}_0$ represents the standard complex structure on $\mathbb{CP}^1$ and
\[
\mathbf{I}_{\zeta_0} = \frac{1}{1 + \zeta_0\bar{\zeta}_0}\left((1 - \zeta_0\bar{\zeta}_0)\mathbf{I} + (\zeta_0 + \bar{\zeta}_0)\mathbf{J} + \left(\sqrt{-1}(-\zeta_0 + \bar{\zeta}_0)\right)\mathbf{K}\right)
\]
In this context, the complex structure on the fiber $\mathcal{Z}_0$ over $0 \in \mathbb{CP}^1$ precisely corresponds to $\mathbf{I}$ as defined above. Furthermore, we can express the complex structure $\mathbf{I}_{\zeta_0}$ on $\mathcal{Z}_{\zeta_0}$ for $\zeta_0 \in U_0$ as a matrix $\mathcal{I}(\zeta_0)$ with respect to the frame $\left(\frac{\partial}{\partial w^1}, \frac{\partial}{\partial w^2}, \frac{\partial}{\partial \bar{w}^1}, \frac{\partial}{\partial \bar{w}^2}\right)$, where $\mathcal{I}(\zeta_0)$ is equal to
\[
\frac{1}{1 + \zeta_0\bar{\zeta}_0}\begin{pmatrix}
\sqrt{-1} (1 - \zeta_0\bar{\zeta}_0) & & & 2\zeta_0 \\
& \sqrt{-1} (1 - \zeta_0\bar{\zeta}_0) & -2\zeta_0 & \\
& 2\bar{\zeta}_0 & \sqrt{-1} (1 - \zeta_0\bar{\zeta}_0) & \\
-2\bar{\zeta}_0 & & & \sqrt{-1} (1 - \zeta_0\bar{\zeta}_0)
\end{pmatrix}
\]

Now, we establish an isomorphism denoted by $\phi:\bigoplus^2\mathcal{O}(1)\xrightarrow{\sim}\mathcal{Z}$. This isomorphism, represented by $\phi_0$ and $\phi_1$ with respect to the given charts, is defined as follows: For $\phi_0$, we map $U_0 \times \mathbb{C}^2 \rightarrow U_0 \times M$ by
\[
(\zeta_0, z_0^1, z_0^2) \mapsto \left(\zeta_0, \frac{1}{1 + \zeta_0\bar\zeta_0}\left(z_0^1 + \sqrt{-1}\zeta_0\bar z_0^2\right), \frac{1}{1 + \zeta_0\bar\zeta_0}\left(z_0^2 - \sqrt{-1}\zeta_0\bar z_0^1\right)\right)
\]
Similarly, for $\phi_1$, we map $U_1 \times \mathbb{C}^2 \rightarrow U_1 \times M$ by
\[
(\zeta_1, z_1^1, z_1^2) \mapsto \left(\zeta_1, \frac{1}{1 + \zeta_1\bar\zeta_1}\left(\bar\zeta_1z_1^1 + \sqrt{-1}\bar z_1^2\right), \frac{1}{1 + \zeta_1\bar\zeta_1}\left(\bar\zeta_1z_1^2 - \sqrt{-1}\bar z_1^1\right)\right)
\]
The final step of our proof is the verification that $\phi$ is a well-defined bundle isomorphism, demonstrating that $\phi$ enables a holomorphic mapping between $\bigoplus^2\mathcal{O}(1)$ and $\mathcal{Z}$. 

First, consider the scenario where \((\zeta_0, z_0^1, z_0^2)\) and \((\zeta_1, z_1^1, z_1^2)\) represent the same point in \(\bigoplus^2\mathcal{O}(1)\). In this case, \(\phi_0(\zeta_0, z_0^1, z_0^2)\) and \(\phi_1(\zeta_1, z_1^1, z_1^2)\) correspond to the same point in \(\mathcal{Z}\), confirming that \(\phi\) is a smooth map.

Next, we examine the holomorphic property of \(\phi\). Given that the complex structure on \(\mathcal{Z}\) is described by \(\mathbf{I}_\zeta \oplus \mathbf{I}_0\), with \(\mathbf{I}_0\) denoting the standard complex structure on \(\mathbb{CP}^1\), it is necessary to verify that the restriction of \(\phi\) to each fiber \(\mathcal{Z}_\zeta\) constitutes a holomorphic map for every \(\zeta \in \mathbb{CP}^1\). Specifically, for \(\zeta_0 \in U_0\), the restriction of \(\phi_0\) to the fiber \(\{\zeta_0\} \times \mathbb{C}^2\) is expressed as
\[
\phi_0\Big|_{\{\zeta_0\} \times \mathbb{C}^2} : (z_0^1, z_0^2) \mapsto \frac{1}{1 + \zeta_0\bar{\zeta}_0}\Big(z_0^1 + \sqrt{-1}\zeta_0\bar{z}_0^2, z_0^2 - \sqrt{-1}\zeta_0\bar{z}_0^1\Big)
\]
The Jacobian matrix of \(\phi_0|_{\zeta_0 \times \mathbb{C}^2}\), denoted \(J(\zeta_0)\), is defined as
\[
J(\zeta_0) := \frac{1}{1 + \zeta_0\bar{\zeta}_0}\begin{pmatrix}
1 & & & \sqrt{-1}\zeta_0 \\
& 1 & -\sqrt{-1}\zeta_0 & \\
& -\sqrt{-1}\bar{\zeta}_0 & 1 & \\
\sqrt{-1}\bar{\zeta}_0 & & & 1
\end{pmatrix}
\]
After verifying that \(\det(J) = 1\), we conclude that the restriction of \(\phi_0\) to \(\zeta_0 \times \mathbb{C}^2\) is an \(\mathbb{R}\)-linear isomorphism for each \(\zeta_0 \in U_0\). Moreover, \(\phi_0|_{\zeta_0 \times \mathbb{C}^2}\) is holomorphic since
\[
\mathcal{I}(\zeta_0) \cdot J(\zeta_0) = J(\zeta_0) \cdot \begin{pmatrix}
i & & & \\
  & i & & \\
  & & -i & \\
  & & & -i
\end{pmatrix}
\]
This indicates that $\mathbf{I}_{\zeta_0} \circ (\phi_0)_* = (\phi_0)_* \circ \mathbf{I}$ for each $\zeta_0 \in U_0$, meaning the restriction of $\phi$ to each fiber $\mathcal{Z}_{\zeta_0}$ is holomorphic for every $\zeta_0 \in U_0$.

\end{proof}

In our previous discussion, we explored the twistor space $\bigoplus^2\mathcal{O}(1)$, derived from the hyperk\"ahler manifold $M \cong \mathbb{C}^2$. We now address the inverse problem of constructing a hyperk\"ahler manifold from the holomorphic fiber bundle $\bigoplus^2\mathcal{O}(1)$. A crucial step in this construction is the introduction of a real structure $\tau$ on $\bigoplus^2\mathcal{O}(1)$. This is achieved by considering a point $p = (x, y)$ in $M$, where the section $\zeta \mapsto (\zeta, p)$ forms a real section in $\mathcal{Z}$. Through the isomorphism $\phi:\bigoplus^2\mathcal{O}(1)\xrightarrow{\sim}\mathcal{Z}$, this real section corresponds to a holomorphic section $s$ of $\bigoplus^2\mathcal{O}(1)$, represented by $s_0$ and $s_1$ as:
\[
\begin{cases}
s_0: U_0 \rightarrow U_0 \times \mathbb{C}^2, & \zeta_0 \mapsto \Big(\zeta_0, x - \sqrt{-1}\zeta_0\bar{y}, y + \sqrt{-1}\zeta_0\bar{x}\Big),\\
s_1: U_1 \rightarrow U_1 \times \mathbb{C}^2, & \zeta_1 \mapsto \Big(\zeta_1, \zeta_1x - \sqrt{-1}\bar{y}, \zeta_1y + \sqrt{-1}\bar{x}\Big).
\end{cases}
\]
The map $\tau$ transforms 
\[
\Big(\zeta, x - \sqrt{-1}\zeta\bar{y}, y + \sqrt{-1}\zeta\bar{x}\Big) \in U_0 \times \mathbb{C}^2 \mapsto \Big(-\bar{\zeta}, -\bar{\zeta}x - \sqrt{-1}\bar{y}, -\bar{\zeta}y + \sqrt{-1}\bar{x}\Big) \in U_1 \times \mathbb{C}^2,
\] 
thereby establishing a real structure on $\bigoplus^2\mathcal{O}(1)$. With the real structure $\tau$, the zero section $0_s$ becomes a real section of $\bigoplus^2\mathcal{O}(1)$, as outlined in Definition \ref{real sections}. The normal bundle of $0_s$ coincides with $\bigoplus^2\mathcal{O}(1)$, leading to $H^1(0_s,\bigoplus^2\mathcal{O}(1))=0$. Consequently, we can explore the regular deformation space $\widetilde{\mathscr{M}_{0_s}^{4n}}$ of $0_s$ in the total space of $\bigoplus^2\mathcal{O}(1)$. In this context, we can construct an analytic family of compact submanifolds, as per Definition \ref{analytic family of compact submanifolds}. With respect to the trivialization of $\bigoplus^2\mathcal{O}(1)$, any holomorphic section $s$ can be denoted as:
\[
\begin{cases}
s_0: U_0 \rightarrow U_0 \times \mathbb{C}^2, & \zeta_0 \mapsto (\zeta_0, a + b\zeta_0),\\
s_1: U_1 \rightarrow U_1 \times \mathbb{C}^2, & \zeta_1 \mapsto (\zeta_1, a\zeta_1 + b),
\end{cases}
\]
and represented by $\varphi(a, b)$. Furthermore, the real section of $\bigoplus^2\mathcal{O}(1)$, corresponding to the real section $\zeta \mapsto (\zeta, x, y)$ of $\mathcal{Z}$ discussed above, can be expressed as $\varphi(x, y, -\sqrt{-1}\bar{y}, \sqrt{-1}\bar{x})$.

We aim to construct an analytic family of compact submanifolds of $\bigoplus^2\mathcal{O}(1)$, each of dimension one and parameterized by $H^0\left(\mathbb{CP}^1, \bigoplus^2\mathcal{O}(1)\right)$, as described in Definition \ref{analytic family of compact submanifolds}. Consider $\mathcal{V} = \bigsqcup_{(a,b) \in \mathbb{C}^4} \varphi(a,b)$ as a submanifold of $\bigoplus^2\mathcal{O}(1) \times H^0\left(\mathbb{CP}^1, \bigoplus^2\mathcal{O}(1)\right)$, which has a complex dimension of five. The natural projection $\pi: \mathcal{V} \rightarrow H^0\left(\mathbb{CP}^1, \bigoplus^2\mathcal{O}(1)\right)$ constitutes the analytic family we seek.

\begin{lemma}
\label {normal bundle in example}
Let $s = \varphi(a^1, a^2, b^1, b^2)$ be any holomorphic section of $\bigoplus^2\mathcal{O}(1)$, then the normal bundle $T_F s$ of $s(\mathbb{CP}^1)$ is isomorphic to $\bigoplus^2\mathcal{O}(1)$.
\end{lemma}

\begin{proof}
Consider the open charts $\big(\mathscr{U}_i := U_i \times \mathbb{C}^2; \zeta_i, z_i^1, z_i^2\big)$ on $\bigoplus^2\mathcal{O}(1)$. The transition functions are given by
\[\begin{cases}
\zeta_0 \cdot \zeta_1 = 1,\\
z_0^1 = \zeta_0 \cdot z_1^1, & z_0^2 = \zeta_0 \cdot z_1^2.
\end{cases}\]
By the definition of $s$ above,
\[\begin{cases}
s \cap \mathscr{U}_0 \text{ is represented by } & z_0^1 = a^1 + b^1\zeta_0 \quad z_0^2 = a^2 + b^2\zeta_0,\\
s \cap \mathscr{U}_1 \text{ is represented by } & z_1^1 = a^1 \zeta_1 + b^1 \quad z_1^2 =  a^2 \zeta_1 + b^2.
\end{cases}\]
Introduce alternative local coordinates $(\eta_0,\tilde{z}_0^1,\tilde{z}_0^2)$ on $\mathscr{U}_0$ by setting $\eta_0 = \zeta_0$ and $\tilde{z}_0^\alpha = z_0^\alpha - (a^\alpha + b^\alpha\zeta_0)$ for $\alpha=1,2$. Similarly, on $\mathscr{U}_1$, set $\eta_1 = \zeta_1$ and let $\tilde{z}_1^\alpha = z_1^\alpha - (a^\alpha \zeta_1 + b^\alpha)$, where $\alpha=1,2$. Let $\big(\mathscr{V}_i = \mathscr{U}_i \cap s(\mathbb{CP}^1); \tilde{z}_i^1, \tilde{z}_i^2\big)$ be open charts on $s(\mathbb{CP}^1)$. Then, the normal bundle $T_Fs$ of $s(\mathbb{CP}^1)$ is represented by a transformation matrix from $\mathscr{V}_0$ to $\mathscr{V}_1$, as defined in Definition \ref{normal bundle},
\[
B(\zeta) = \left( \left. \frac{\partial \tilde{z}_1^\alpha}{\partial \tilde{z}_0^\beta} \right|_{s(\mathbb{CP}^1)} \right)_{\alpha,\beta=1,2} = \left(
\begin{array}{cc}
\zeta_0^{-1} &  \\
 & \zeta_0^{-1}
\end{array}
\right).
\]
which implies that $T_Fs$ is isomorphic to $\bigoplus^2\mathcal{O}(1)$.
\end{proof}

By Lemma \ref{normal bundle in example} above, $\widetilde{\mathscr{M}_{0_s}^{4n}}$ is precisely the parameter space $H^0\left(\mathbb{CP}^1, \bigoplus^2\mathcal{O}(1)\right) \cong \mathbb{C}^4$, since $H^1(s(\mathbb{CP}^1), T_Fs) = 0$ for each holomorphic section $s$ of $\bigoplus^2\mathcal{O}(1)$, and moreover, we consider
\[
\mathcal{M} := \left\{ V \in \widetilde{\mathscr{M}_{0_s}^{4n}} \Big| V \text{ is a holomorphic section of } \bigoplus^2\mathcal{O}(1) \text{ and } T_FV \cong \bigoplus^2\mathcal{O}(1) \right\}
\]
to be precisely the entire set $\pi^{-1}\Big(H^0\left(\mathbb{CP}^1, \bigoplus^2\mathcal{O}(1)\right)\Big) = \left\{ \varphi(a, b) | (a, b) \in \mathbb{C}^4 \right\}$. In summary, we have the following equation:
\[\mathcal{M} = \widetilde{\mathscr{M}_{0_s}^{4n}} = H^0\left(\mathbb{CP}^1, \bigoplus^2\mathcal{O}(1)\right) \cong \mathbb{C}^4.\]

A holomorphic section $\varphi(a^1, a^2, b^1, b^2)$ is a real section if and only if $(b^1, b^2)$ equals $(-\sqrt{-1}\bar{a}^2, \sqrt{-1}\bar{a}^1)$. Based on this criterion, we can view $\mathcal{M}$ as the complex vector space $\mathbb{C}^4 = (a^1, a^2, b^1, b^2)$, and the hyperkähler manifold $M$ can be considered as a smooth submanifold of $\mathcal{M}$, represented as the set 
\[
\Big\{ (a^1, a^2, -\sqrt{-1}\bar{a}^2, \sqrt{-1}\bar{a}^1) \mid (a^1, a^2) \in \mathbb{C}^2 \Big\}.
\]
This formulation captures the conditions under which a holomorphic section is also a real section, thereby characterizing the elements of $\mathcal{M}$ that correspond to points on the hyperkähler manifold $M$.


\bibliographystyle{plain}
\bibliography{RefBase}

\end{document}